\documentclass[11pt]{amsart}

\usepackage[margin=1in]{geometry}
\usepackage[T1]{fontenc}
\usepackage[utf8]{inputenc}
\usepackage{amsmath,amssymb,amsthm}
\usepackage{xcolor}
\usepackage{tikz}
\usepackage{aliascnt}
\usepackage{hyperref}
\usepackage{cleveref}

\hypersetup{
  colorlinks=true,
  linkcolor=blue!60!black,
  citecolor=blue!60!black,
  urlcolor=blue!60!black
}

\allowdisplaybreaks
\numberwithin{equation}{section}

\newtheorem{theorem}{Theorem}[section]
\newaliascnt{proposition}{theorem}
\newtheorem{proposition}[proposition]{Proposition}
\aliascntresetthe{proposition}
\newaliascnt{lemma}{theorem}
\newtheorem{lemma}[lemma]{Lemma}
\aliascntresetthe{lemma}
\newaliascnt{corollary}{theorem}
\newtheorem{corollary}[corollary]{Corollary}
\aliascntresetthe{corollary}
\newaliascnt{conjecture}{theorem}
\newtheorem{conjecture}[conjecture]{Conjecture}
\aliascntresetthe{conjecture}
\theoremstyle{definition}
\newaliascnt{definition}{theorem}
\newtheorem{definition}[definition]{Definition}
\aliascntresetthe{definition}
\newaliascnt{example}{theorem}
\newtheorem{example}[example]{Example}
\aliascntresetthe{example}
\theoremstyle{remark}
\newaliascnt{remark}{theorem}
\newtheorem{remark}[remark]{Remark}
\aliascntresetthe{remark}

\crefname{theorem}{Theorem}{Theorems}
\crefname{proposition}{Proposition}{Propositions}
\crefname{lemma}{Lemma}{Lemmas}
\crefname{corollary}{Corollary}{Corollaries}
\crefname{conjecture}{Conjecture}{Conjectures}
\crefname{definition}{Definition}{Definitions}
\crefname{example}{Example}{Examples}
\crefname{remark}{Remark}{Remarks}
\crefname{figure}{Figure}{Figures}
\Crefname{theorem}{Theorem}{Theorems}
\Crefname{proposition}{Proposition}{Propositions}
\Crefname{lemma}{Lemma}{Lemmas}
\Crefname{corollary}{Corollary}{Corollaries}
\Crefname{conjecture}{Conjecture}{Conjectures}
\Crefname{definition}{Definition}{Definitions}
\Crefname{example}{Example}{Examples}
\Crefname{remark}{Remark}{Remarks}
\Crefname{figure}{Figure}{Figures}
\crefformat{equation}{(#2#1#3)}
\Crefformat{equation}{Equation #2(#1)#3}

\newcommand\wt{\operatorname{wt}}
\newcommand\Par{\operatorname{Par}}

\newcommand{\qbinom}[2]{\genfrac{[}{]}{0pt}{}{#1}{#2}_{q}}
\newcommand{\qp}[2]{(#1;q)_{#2}}

\newcommand{\PAW}{P^{\mathrm{AW}}}
\newcommand{\PZ}{P^{Z}}

\tikzset{
  awlane/.style={draw=purple!70!black,line width=1.15pt,line cap=round},
  awstep/.style={draw=red!70!black,line width=1.05pt,line cap=round},
  awlabel/.style={text=red!70!black,font=\scriptsize,inner sep=1pt},
  awendpoint/.style={text=teal!60!black,font=\large\bfseries,inner sep=2pt}
}

\title{Askey--Wilson polynomials with ASEP parameters}

\author{Younggwang Cho}
\address{Department of Mathematics, Sungkyunkwan University, Suwon, South Korea}
\email{brglory@g.skku.edu}

\author{Donghyun Kim}
\address{Department of Mathematics, Ewha Womans University, Seoul, South Korea}
\email{kdh310@ewha.ac.kr}

\author{Jang Soo Kim}
\address{Department of Mathematics, Sungkyunkwan University, Suwon, South Korea}
\email{jangsookim@skku.edu}

\thanks{D.~H.~Kim was supported by the Ewha Womans University Research Grant of 2026.
Y.~Cho and J.~S.~Kim were supported by the National Research Foundation of Korea (NRF) grant funded by the Korea government (RS-2025-00557835).}


\begin{document}

\begin{abstract}
Koornwinder moments generalize the Askey--Wilson moments arising in the asymmetric simple exclusion process. Rains conjectured that, under the specialization $t=q$, the minimal numerators of Koornwinder moments have nonnegative integer coefficients. While the one-row case of this conjecture was previously proved by Corteel, Mandelshtam, and Williams using rhombic staircase tableaux, its dual counterpart, the one-column case, has remained open. In this paper, we derive a closed, manifestly positive combinatorial formula for the normalized numerators of the coefficients of the rescaled Askey--Wilson polynomials. This proves Rains' conjecture for one-column partitions, thereby providing the exact dual counterpart to the previous result.
\end{abstract}

\maketitle

\section{Introduction}

The interplay between particle systems arising in statistical
mechanics and algebraic combinatorics has seen remarkable progress
over the past three decades. A celebrated milestone in this development
is the exact solvability of the asymmetric simple exclusion process
(ASEP) with open boundaries. In the pioneering work of Sasamoto
\cite{Sasamoto1999} and Uchiyama, Sasamoto, and Wadati \cite{USW2004},
the partition function of the one-species
ASEP was shown to be governed by the moments of Askey--Wilson
polynomials, the master family of orthogonal polynomials sitting at
the top of the hierarchy of basic hypergeometric orthogonal
polynomials \cite{GR, KLS}. This algebraic connection was subsequently
enriched with deep combinatorial meaning by Corteel and Williams
\cite{Corteel2011}, who introduced staircase tableaux to give a
combinatorial interpretation for the steady-state probabilities of the
system.

For multi-species generalizations of the exclusion process, the
relevant orthogonal families naturally lift from Askey--Wilson
polynomials to Koornwinder polynomials, also called
Macdonald--Koornwinder polynomials of type \(BC\)
\cite{Koornwinder1992}, which are multivariate orthogonal polynomials
with parameters \( a,b,c,d,q \) and \( t \). To study the two-species
ASEP on a line, Corteel and Williams \cite{Corteel2019} introduced the
\emph{Koornwinder moments}
\(M^Z_\lambda = M^Z_\lambda(\xi; \alpha, \beta, \gamma, \delta; q) \)
under the specialization \(t=q\). These moments are rational functions
in \( \xi,\alpha,\beta,\gamma,\delta,q \) and are indexed by
partitions \(\lambda\) whose parts are allowed to be zero. This
convention is physically significant: while \(M^Z_{(n)}\) yields the
partition function \(Z_n\) of the classical ASEP, \(M^Z_{(n-k,0^k)}\)
yields the partition function \(Z_{n,k}\) of the two-species ASEP on
\(n\) sites with light and heavy particles, exactly \(k\) of which are
light. For more details on the multi-species ASEP and Koornwinder
polynomials, we refer to \cite{Cantini2017}.

Corteel, Mandelshtam, and Williams \cite{CMW} gave a combinatorial
interpretation of \(Z_{n,k}\) using rhombic staircase tableaux. This
interpretation implies that the minimal numerator of
\(M^Z_{(n-k,0^k)}\) is a polynomial in
\(\xi,\alpha,\beta,\gamma,\delta,q\) with nonnegative integer
coefficients. Generalizing this result, Rains \cite[Conjecture
4.4]{Corteel2019} formulated a positivity conjecture: the minimal
numerator of any Koornwinder moment \(M^Z_\lambda\) is a polynomial in
the parameters \(\xi,\alpha,\beta,\gamma,\delta,q\) with nonnegative
integer coefficients.

The following cases of Rains' conjecture are known:
\begin{itemize}
\item Rains' conjecture was proved for \emph{one-row partitions}
  \(\lambda = (n-k,0^k)\) by Corteel, Mandelshtam, and Williams
  \cite{CMW} using rhombic staircase tableaux. The case without zero
  parts, \(\lambda=(n)\), had been proved earlier using staircase
  tableaux in \cite{Corteel2011}.
\item For an arbitrary partition \(\lambda\), Rains' conjecture
  has recently been established under the parameter specializations
  \((\xi,q)=(1,0)\) and \((\xi,q)=(1,1)\) in \cite{Cho2026b}.
\end{itemize}

In this paper, we resolve Rains' conjecture for \emph{one-column
  partitions} \(\lambda = (1^{n-k},0^k)\), thereby providing the
natural dual counterpart to the result of Corteel, Mandelshtam, and
Williams \cite{CMW} for one-row partitions.

The Koornwinder moment \(M^Z_{(n-k,0^k)}\) for a one-row partition can
also be understood as the corresponding mixed moment of the
rescaled Askey--Wilson polynomial. Similarly, the Koornwinder moment
\(M^Z_{(1^{n-k},0^k)}\) for a one-column partition can be understood as
the corresponding coefficient of the rescaled Askey--Wilson polynomial up to
sign. Hence,
our main object of study is the normalized numerator
\(\widetilde{\nu}^{Z}_{n,k}\) of the coefficient of \(x^k\) in the
rescaled Askey--Wilson polynomial.

Our main result is a manifestly positive formula for
\( \widetilde{\nu}^{Z}_{n,k} \) using lattice paths
(\Cref{thm:ASEP-positivity}). This implies that Rains' conjecture
holds for one-column partitions.

A principal reason that the formula in \Cref{thm:ASEP-positivity} was
difficult to find is that these numerators
\( \widetilde{\nu}^{Z}_{n,k} \) do not appear to arise from a single
global weighted lattice that works simultaneously for every degree
\(n\). Our solution is to use a different finite lattice-path model
\(\mathcal{G}_n\) for each \(n\): the edge weights depend
intrinsically on \(n\), and for each fixed \(n\), paths in the graph
\(\mathcal{G}_n\) give \(\widetilde{\nu}_{n,k}^Z\) for all
\(0\le k\le n\). In \Cref{sec:appendix}, we explain how we obtained
the lattice path model.

\begin{remark}
  Our result is different from, and independent of, the lattice-path
  formula in \cite{Corteel2026a}. First, in \cite{Corteel2026a}, the
  authors study the standard monic Askey--Wilson polynomials
  \(\PAW_n(x)\), whereas we study the rescaled Askey--Wilson
  polynomials \(\PZ_n(x)\) tailored to the two-species ASEP. Second,
  the coefficients of \(\PAW_n(x)\) are rational functions in
  \(a,b,c,d,q\), and the formula in \cite{Corteel2026a} uses infinite
  lattice paths to represent their series expansions. By contrast, our
  model is finite and directly describes the minimal numerator
  \( \widetilde{\nu}^{Z}_{n,k} \) of the coefficient.
\end{remark}

The rest of this paper is organized as follows. In
\Cref{sec:preliminaries}, we provide necessary definitions, review the
connection between ASEP and Askey--Wilson moments and their
generalization to two-species ASEP and Koornwinder moments. In
\Cref{sec:comb-formula}, we state our main theorem, a combinatorial
model for \( \widetilde{\nu}^{Z}_{n,k} \). In \Cref{sec:proof-main},
we prove the main theorem. Finally, \Cref{sec:appendix} describes the
heuristic process behind the discovery of our formula.

\subsection*{Acknowledgments}
The second author is grateful to Sylvie Corteel and Lauren Williams
for introducing him to the positivity problem for the coefficients of
rescaled Askey--Wilson polynomials. The authors would like to thank
Minho Song for helpful discussions.

\subsection*{Declaration on Generative AI}

The authors used the ChatGPT 5.5, 5.6, and 6 Pro models to assist in
finding and proving \Cref{lem:inner-sum} and \Cref{lem:outer-sum},
drawing diagrams, writing Sage code, and polishing the paper. All
other mathematical content is solely the work of the authors. In
particular, the formula in the main thereom
(Theorem~\ref{thm:ASEP-positivity}) was originally discovered by the
second author in 2020 without using any AI tools.

\section{Preliminaries}\label{sec:preliminaries}

In this section, we introduce basic definitions, review previous
results, and state Rains' conjecture.

A \emph{partition} is a weakly decreasing sequence
\( \lambda=(\lambda_1,\dots,\lambda_\ell) \) of positive integers.
Each \( \lambda_i \) is called a \emph{part}. Let \( \Par_m \) denote
the set of partitions with at most \( m \) parts. When
\( \lambda \in \Par_m \), we regard it as the \( m \)-tuple
\( \lambda=(\lambda_1,\dots,\lambda_m) \) obtained by appending zero
parts if necessary. We write \(a^k\) for \(k\) consecutive copies of
\(a\). For example, \( (4,1,1) =(4,1,1,0,0,0,0)= (4,1^2,0^4)\in\Par_7 \).

The \emph{Young diagram} of a partition
\( \lambda=(\lambda_1,\dots,\lambda_\ell) \) is a left-justified array
of unit squares in which there are \( \lambda_i \) squares in the
\( i \)th row from the top. For example, the Young diagram of
\( \lambda=(4,3,1) \) is drawn as follows.
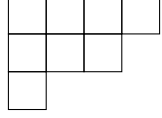
\begin{center}
    \begin{tikzpicture}[scale=0.5]
    \foreach \x in {0,1,2,3} {
      \draw (\x,0) rectangle ++(1,-1);
    }
    \foreach \x in {0,1,2} {
      \draw (\x,-1) rectangle ++(1,-1);
    }
    \draw (0,-2) rectangle ++(1,-1);
  \end{tikzpicture}
\end{center}
In view of this, a \emph{one-row partition} is a partition of the form
\( (r) \) and a \emph{one-column partition} is a partition of the form
\( (1^r) \).

We use the following notation for \( q \)-series:
\begin{equation*}
 (x;q)_r = \prod_{i=0}^{r-1}(1-xq^i),\quad
 (x_1,\ldots,x_m;q)_r = \prod_{i=1}^{m}(x_i;q)_r, \quad
 [r]_q=\frac{1-q^r}{1-q},\quad
\qbinom{r}{s} = \frac{(q;q)_r}{(q;q)_s (q;q)_{r-s}},
\end{equation*}
where \( (x;q)_0 = 1 \) and \( \qbinom{r}{s}=0 \) unless
\( 0\le s\le r \). For integers \(r,s\geq0\), the basic hypergeometric
series is defined by
\begin{equation*}
 {}_r\phi_s\!\left[
 \begin{matrix}a_1,\ldots,a_r\\ b_1,\ldots,b_s\end{matrix}
 ;q,z
 \right]
 =\sum_{k=0}^{\infty}
 \frac{(a_1,\ldots,a_r;q)_k}{(q,b_1,\ldots,b_s;q)_k}
 \left((-1)^kq^{\binom{k}{2}}\right)^{1+s-r}z^k.
\end{equation*}
We refer the reader to \cite{GR} for basic properties of
these objects.

In this paper, we use the following standard monic normalization of
the Askey--Wilson polynomials \cite[Section~14.1]{KLS}.

\begin{definition}\label{def:AW-recurrence}
 Let \(z\) be an auxiliary variable with \(x=(z+z^{-1})/2\). For
 \(n\geq0\), the \emph{monic Askey--Wilson polynomial}
 \(\PAW_n(x)=\PAW_n(x;a,b,c,d|q)\) is defined by
 \begin{equation*}
 \PAW_n(x;a,b,c,d|q)
 =\frac{\qp{ab,ac,ad}{n}}
 {(2a)^n\qp{abcdq^{n-1}}{n}}
 {}_4\phi_3\!\left[
 \begin{matrix}
 q^{-n},\,abcdq^{n-1},\,az,\,az^{-1}\\
 ab,\,ac,\,ad
 \end{matrix}
 ;q,q
 \right].
 \end{equation*}
\end{definition}

The work by Sasamoto \cite{Sasamoto1999} and Uchiyama,
Sasamoto, and Wadati \cite{USW2004} revealed that the partition
function \( Z_n(\alpha,\beta,\gamma,\delta;q) \) for the asymmetric
simple exclusion process (ASEP) on a line can be computed via the
moments of Askey--Wilson polynomials. A central combinatorial feature
of this connection is the positivity of the moments: the minimal
numerator of the fugacity partition function
\( Z_n(\xi;\alpha,\beta,\gamma,\delta;q) \) of the ASEP, with the
additional fugacity parameter \( \xi \), is a polynomial in
\( \xi,\alpha,\beta,\gamma,\delta,q \) with nonnegative coefficients.
This was first established by Corteel and Williams \cite{Corteel2011}
using staircase tableaux.

More generally, Corteel and Williams \cite{Corteel2018} showed that
\( Z_{n,k}(\xi;\alpha,\beta,\gamma,\delta;q) \), the fugacity partition
function of the two-species ASEP, is a Koornwinder moment indexed by
a one-row partition. Corteel, Mandelshtam, and Williams \cite{CMW} gave a
combinatorial model for this partition function using rhombic
staircase tableaux, thereby showing that its minimal numerator is a
polynomial with nonnegative coefficients.

To explain the above results more precisely, we follow the notation of
\cite{Cho2026b}. First, consider the following reparametrization
relating the Askey--Wilson parameters \(a,b,c,d\) to the ASEP
parameters \(\alpha,\beta,\gamma,\delta\).

\begin{definition}\cite[p.~395]{Corteel2011}
  \label{def:parameter-change}
Let
 \begin{align*}
   \alpha &= \frac{1-q}{(1+a)(1+c)}, & \beta &= \frac{1-q}{(1+b)(1+d)},\\
   \gamma &= -\frac{(1-q)ac}{(1+a)(1+c)}, & \delta &= -\frac{(1-q)bd}{(1+b)(1+d)}.
 \end{align*}
Equivalently,
 \begin{align*}
  a&=\frac{1-q-\alpha+\gamma
    +\sqrt{(1-q-\alpha+\gamma)^2+4\alpha\gamma}}{2\alpha},\\
  b&=\frac{1-q-\beta+\delta
    +\sqrt{(1-q-\beta+\delta)^2+4\beta\delta}}{2\beta},\\
  c&=\frac{1-q-\alpha+\gamma
    -\sqrt{(1-q-\alpha+\gamma)^2+4\alpha\gamma}}{2\alpha},\\
  d&=\frac{1-q-\beta+\delta
    -\sqrt{(1-q-\beta+\delta)^2+4\beta\delta}}{2\beta}.
 \end{align*}
\end{definition}

Note that we have
\begin{equation}\label{eq:abcd-ASEP}
  abcd = \frac{\gamma\delta}{\alpha\beta},
  \qquad
  1 - abcdq^r = \frac{\alpha\beta - \gamma\delta q^r}{\alpha\beta}.
\end{equation}
Using the parameters \( \xi, \alpha, \beta, \gamma, \delta, q \), we
introduce the rescaled Askey--Wilson polynomials.

\begin{definition}\label{def:Z-recurrence}
  The \emph{rescaled Askey--Wilson polynomials}
  \(P^Z_n(x) = P^Z_n(x;\xi;\alpha,\beta,\gamma,\delta | q)\) are
  defined by
  \begin{equation} \label{eq:PZ-relation}
 \PZ_n(x;\xi;\alpha,\beta,\gamma,\delta | q)
 =
 \left(\frac{2\sqrt{\xi}}{1-q}\right)^n
 \PAW_n\left(
   \frac{(1-q)x-(1+\xi)}{2\sqrt{\xi}};
   \frac{a}{\sqrt{\xi}}, \, b\sqrt{\xi}, \,
   \frac{c}{\sqrt{\xi}}, \, d\sqrt{\xi}
   \,\middle|\,q\right),
\end{equation}
where the parameters \( a,b,c,d \) on the right-hand side are
rewritten in terms of \( \alpha,\beta,\gamma,\delta \) using the
formulas in \Cref{def:parameter-change}.
\end{definition}

The \emph{coefficients}
\( \nu^Z_{n,k} = \nu^Z_{n,k}(\xi; \alpha, \beta, \gamma, \delta; q) \)
and the \emph{mixed moments}
\( \sigma^Z_{n,k} =\sigma^Z_{n,k}(\xi; \alpha, \beta, \gamma, \delta;
q) \) of the rescaled Askey--Wilson polynomial are defined by
\[
  \PZ_n(x) = \sum_{k=0}^n \nu_{n, k}^Z x^k, \qquad 
  x^n = \sum_{k=0}^n \sigma_{n, k}^Z \PZ_k(x).
\]
We also define the \emph{moments} \( \sigma^Z_n := \sigma^Z_{n,0} \).
Note that the coefficients and the mixed moments are dual to each
other in the sense that their matrices \( (\nu^Z_{i,j})_{i,j=0}^n \)
and \( (\sigma^Z_{i,j})_{i,j=0}^n \) are mutually inverse. The
results in \cite[Theorem~1.11]{Corteel2012} and
\cite[Corollary~6.2]{Corteel2019} give the following connections
between the fugacity partition functions and the mixed moments (see
also \cite[Theorems~3.5 and 3.6]{Cho2026b}):
\begin{align*}
Z_{n}(\xi;\alpha,\beta,\gamma,\delta;q)
&= \sigma_{n}^Z(\xi;\alpha,\beta,\gamma,\delta;q),\\
Z_{n,k}(\xi;\alpha,\beta,\gamma,\delta;q)
&= \sigma_{n,k}^Z(\xi;\alpha,\beta,\gamma,\delta;q).
\end{align*}

For a partition \( \lambda\in \Par_m \), the \emph{Koornwinder moment}
\(M^Z_\lambda = M^Z_\lambda(\xi; \alpha, \beta, \gamma, \delta; q)\)
is defined by
\[
  M^Z_\lambda = \det(\sigma^Z_{\lambda_i+m-i,m-j})_{i,j=1}^m.
\]
By \cite[Corollary~2.10]{Cho2026b} and \cite[Lemma~2.14]{Cho2026b},
both mixed moments and coefficients of rescaled Askey--Wilson
polynomials are special Koornwinder moments up to sign:
\begin{equation}\label{eq:1}
  \sigma^Z_{n,k} = M^Z_{(n-k,0^{k})}, \qquad 
   \nu_{n,k}^Z = (-1)^{n-k} M^{Z}_{(1^{n-k},0^k)} .
\end{equation}
Therefore, \( Z_{n,k}=\sigma^Z_{n,k} \) is a Koornwinder moment
indexed by a one-row partition, and its minimal numerator has
nonnegative coefficients. Generalizing this positivity phenomenon,
Rains conjectured that the minimal numerator of \( M^{Z}_{\lambda} \)
is a polynomial with nonnegative integer coefficients; see
\cite[Conjecture~4.4]{Corteel2019}. Cho et
al.~\cite[Proposition~6.4]{Cho2026b} found an explicit formula for the
minimal denominator of \( M^{Z}_{\lambda} \). Using their formula, we
can state Rains' conjecture as follows.

\begin{conjecture}[Rains' conjecture] 
  For a partition \( \lambda \in \Par_m \), the following is a
  polynomial in \( \xi, \alpha, \beta, \gamma, \delta, q \) with
  nonnegative integer coefficients:
  \begin{equation}\label{eq:5}
    M^{Z}_{\lambda}(\xi; \alpha, \beta, \gamma, \delta; q)
    \prod_{i=1}^m \prod_{j=1}^{\lambda_i}
      \left(\alpha\beta-\gamma\delta q^{2m-2-i+j}\right).
  \end{equation}
\end{conjecture}

Using this setting, the results in Corteel and Williams
\cite{Corteel2011} and Corteel, Mandelshtam, and Williams \cite{CMW}
imply that Rains' conjecture is true for \( \lambda=(n) \) and
\( \lambda=(n-k,0^k) \), respectively.

In this paper, we prove Rains' conjecture for one-column partitions
\( \lambda=(1^{n-k},0^k) \). This case corresponds to the coefficients
of rescaled Askey--Wilson polynomials; hence, our result is dual to
that of Corteel, Mandelshtam, and Williams \cite{CMW} on the mixed
moments.

We define the normalized numerator
\( \widetilde{\nu}_{n, k}^Z = \widetilde{\nu}_{n, k}^Z(\xi; \alpha,
\beta, \gamma, \delta; q) \) of the coefficient \( \nu_{n,k}^Z \) by
\begin{equation}\label{eq:nu_tilde}
 \widetilde{\nu}_{n, k}^Z := (-1)^{n-k}\nu_{n, k}^Z
 \prod_{r=0}^{n-k-1}(\alpha\beta - \gamma\delta q^{n+k-1+r}) .
\end{equation}
By \eqref{eq:1}, \( \widetilde{\nu}_{n,k}^Z \) is exactly the
expression in \eqref{eq:5} when
\( \lambda=(1^{n-k},0^k) \). Hence, in this case, Rains' conjecture
states that \( \widetilde{\nu}_{n, k}^Z \) is a polynomial with
nonnegative integer coefficients.

Our primary objective is to establish a closed, manifestly positive
combinatorial formula for \(\widetilde{\nu}_{n, k}^Z\)
(\Cref{thm:ASEP-positivity}), thereby settling Rains' conjecture for
all one-column partitions.

\begin{example}\label{ex:small-PZ}
We list the first few rescaled Askey--Wilson polynomials:
\begin{align*}
  \PZ_0(x) &= 1, \\[1ex]
  \PZ_1(x) &= x - \frac{\widetilde{\nu}_{1,0}^Z}{\alpha\beta - \gamma\delta}, \\[1ex]
  \PZ_2(x) &= x^2 - \frac{\widetilde{\nu}_{2,1}^Z}{\alpha\beta - \gamma\delta q^2}\, x
              + \frac{\widetilde{\nu}_{2,0}^Z}{(\alpha\beta - \gamma\delta q)(\alpha\beta - \gamma\delta q^2)},
\end{align*}
where
\begin{align*}
  \widetilde{\nu}_{1,0}^Z &= (\alpha + \delta)\xi + \beta + \gamma, \\[1ex]
  \widetilde{\nu}_{2,1}^Z &= (\alpha\beta + \alpha\delta + \alpha + \alpha\delta q + \gamma\delta q + \alpha q + \delta q + \delta q^2)\xi \\
          &\quad + (\alpha\beta + \beta\gamma + \beta + \beta \gamma q + \gamma\delta q + \beta q + \gamma q + \gamma q^2), \\[1.5ex]
  \widetilde{\nu}_{2,0}^Z &= \bigl(\alpha^2\beta + \alpha^2 q + \alpha\delta q + \alpha^2\delta q + \alpha\delta^2 q + \alpha\delta q^2 + \gamma\delta^2 q^2 + \delta^2 q^2\bigr)\xi^2 \\
          &\quad + (1+q)\bigl(\alpha\beta(1 + \gamma + \delta) + (\alpha\gamma + \beta\delta)q + \gamma\delta q(q + \alpha + \beta)\bigr)\xi \\
          &\quad + \bigl(\alpha\beta^2 + \beta^2 q + \beta\gamma q + \beta^2\gamma q + \beta\gamma^2 q + \beta\gamma q^2 + \gamma^2\delta q^2 + \gamma^2 q^2\bigr).
\end{align*}
\end{example}

\section{A combinatorial formula for \( \widetilde{\nu}_{n, k}^Z \)}
\label{sec:comb-formula}

In this section, we state our main result: a combinatorial formula for
\( \widetilde{\nu}_{n, k}^Z \) (\Cref{thm:ASEP-positivity}). This
implies Rains' conjecture for one-column partitions.

\subsection{Main result}

Let \(\mathcal{G}_n\) be the directed graph whose vertex set \( V \)
and edge set \( E \) are given by
\begin{align*}
  V &= \bigl\{(r,s)\in\mathbb{Z}^2:0\le r\le n,\ 0\le s\le 2r\bigr\},\\
  E &= \bigl\{(u,v)\in V^2: u=(r,s), v=(r,s-1)\} \cup
      \bigl\{(u,v)\in V^2: u=(r,s), v=(r+1,s)\}.
\end{align*}
Every south edge \((r,s)\to(r,s-1)\) has weight \(1\), and the east edge \((r,s)\to(r+1,s)\) has weight
\begin{equation}\label{eq:W-ASEP-def}
  w_n^Z(r,s):=
  \begin{cases}
    (\beta+\xi\delta q^{n-r+2t-1})
      \bigl((\alpha+\gamma q^{r-t})[r-t]_q+q^{r-t}\bigr),
      & \text{if \(s=2t\)},\\
    (\beta+\delta q^{n-r+2t})
      (\xi\alpha+\gamma q^{r-t-1})[r-t]_q,
      & \text{if \(s=2t+1\)}.
  \end{cases}
\end{equation}
For vertices \(u\) and \(v\) of \(\mathcal{G}_n\), let
\(\mathcal{P}(u \to v)\) denote the set of paths in \(\mathcal{G}_n\) from
\(u\) to \(v\). For \(\pi\in\mathcal{P}(u \to v)\), define
\(\wt_n^Z(\pi)\) to be the product of its edge weights.

\begin{figure}
  \centering
  \begin{tikzpicture}[x=1.45cm,y=0.55cm]
  \def\N{6}
  \pgfmathtruncatemacro{\LastColumn}{\N-1}

  \fill (0,0) circle (2.2pt)
    node[awendpoint,below left=2pt] {\(u_0=v_0\)};
  \foreach \k in {1,...,\N} {
    \pgfmathtruncatemacro{\Ysink}{2*\k}
    \draw[awlane] (\k,0) -- (\k,\Ysink);
    \fill (\k,\Ysink) circle (2.2pt)
      node[awendpoint,above left=2pt] {\(u_{\k}\)};
  }

  \foreach \r in {0,...,\LastColumn} {
    \pgfmathtruncatemacro{\Smax}{2*\r}
    \foreach \s in {0,...,\Smax} {
      \draw[awstep] (\r,\s) -- (\r+1,\s)
        node[midway,above=0.5pt,awlabel,font=\scriptsize,fill=white,inner sep=0.3pt]
          {\((\r,\s)\)};
    }
  }

  \draw[draw=blue!70!black,line width=2.4pt,line cap=round,line join=round]
    (2,4) -- (3,4) -- (3,2) -- (4,2) -- (4,1) -- (5,1) -- (5,0);

  \foreach \i in {1,...,\N} {
    \fill (\i,0) circle (2.2pt)
      node[awendpoint,below=2pt] {\(v_{\i}\)};
  }
  \end{tikzpicture}
  \caption{The graph \( \mathcal{G}_6 \), where each east edge is
    labeled by its starting point \((r,s)\). The highlighted path
    \(\pi\) from \(u_2\) to \(v_5\) has weight
    \(\wt_6^Z(\pi)=w_6^Z(2,4)w_6^Z(3,2)w_6^Z(4,1)\).}
  \label{fig:lattice_ya}
\end{figure}
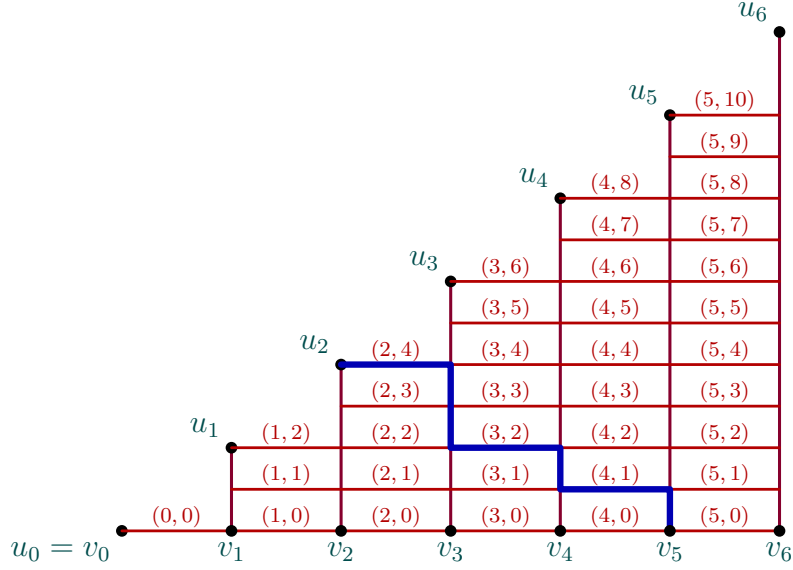

For \( 0\le i\le n \), let \(u_i=(i,2i)\) and \(v_i=(i,0)\), as in
\Cref{fig:lattice_ya}. For \(0\le k\le i\le n\), we define
\begin{equation}\label{eq:G-ASEP-def}
  G_n^Z(i,k)
  :=\sum_{\pi\in\mathcal{P}(u_k \to v_i)}\wt_n^Z(\pi).
\end{equation}
For \(0\le k\le n\), we also define
\begin{equation}\label{eq:T-ASEP-def}
  T^Z_{n,k} := \qbinom{n}{k}
    \prod_{r=0}^{n-k-1}
    \bigl(q^r+[r]_q(\beta+\delta q^r)\bigr)
    (\xi\alpha+\gamma q^{k+r}).
\end{equation}

We are now ready to state our main result, which gives a manifestly
positive formula for \( \widetilde{\nu}_{n, k}^Z \).

\begin{theorem}\label{thm:ASEP-positivity}
  For \(0\le k\le n\), we have
  \begin{equation}\label{eq:ASEP-positive-formula}
    \widetilde{\nu}_{n, k}^Z
    = \sum_{i=k}^{n} T^Z_{n,i} G^{Z}_n(i,k).
  \end{equation}
  Consequently,
  \[
    \widetilde{\nu}_{n, k}^Z \in
    \mathbb{Z}_{\ge0}[\xi, \alpha,\beta,\gamma,\delta,q].
  \]
\end{theorem}

We will prove \Cref{thm:ASEP-positivity} in \Cref{sec:proof-main}. In
\Cref{sec:appendix}, we explain how we discovered this theorem. As
noted in \Cref{sec:preliminaries}, \Cref{thm:ASEP-positivity} implies
Rains' conjecture for one-column partitions.

\begin{corollary}\label{cor:Rains-one-column}
  For \(0\le k\le n\), Rains' conjecture holds for
  \(\lambda=(1^{n-k},0^k)\).
\end{corollary}

It is also possible to express \( \widetilde{\nu}_{n, k}^Z \) as a
single generating function for lattice paths rather than as the sum in
\eqref{eq:ASEP-positive-formula}. To do this, we first find a lattice
path model for \(T^Z_{n,i}\).

For \(0\le r<n\), define
\begin{equation}\label{eq:boundary-column-weight}
  \tau_n^Z(r)
  =\bigl(q^{n-r-1}+[n-r-1]_q(\beta+\delta q^{n-r-1})\bigr)
    (\xi\alpha+\gamma q^r).
\end{equation}
Let \(\mathcal{H}_n\) be the directed graph with vertex set \(V_H\) and
edge set \(E_H\) given by
\begin{align*}
  V_H
  &=\bigl\{(r,s)\in\mathbb{Z}^2:0\le r\le n,\ -r\le s\le 2r\bigr\},\\
  E_H
  &=\bigl\{((r,s),(r,s-1)):0\le r\le n,\ -r+1\le s\le 2r\bigr\}\\
  &\quad\cup\bigl\{((r,s),(r+1,s)):0\le r<n,\ 0\le s\le 2r\bigr\}\\
  &\quad\cup\bigl\{((r,s),(r+1,s-1)):0\le r<n,\ -r\le s\le0\bigr\}.
\end{align*}
Every south edge \((r,s)\to(r,s-1)\) has weight \(1\). For
\(0\le r<n\) and \(0\le s\le 2r\), every east edge \((r,s)\to(r+1,s)\)
has weight \(w_n^Z(r,s)\). For \(0\le r<n\) and \(-r\le s\le0\), every
southeast edge \((r,s)\to(r+1,s-1)\) has weight
\(\tau_n^Z(r)q^{r+s}\). Thus the upper part is \(\mathcal{G}_n\),
while paths in the lower part use south and southeast steps.

For vertices \(u,v\in V_H\), let \(\mathcal{P}(u\to v)\) denote the
set of directed paths in \(\mathcal{H}_n\) from \(u\) to \(v\). For
\(0\le i\le n\), let \(d_i=(i,-i)\). For a path \(\eta\), let
\(\wt(\eta)\) be the product of its edge weights. See
\Cref{fig:combined-lattice}.

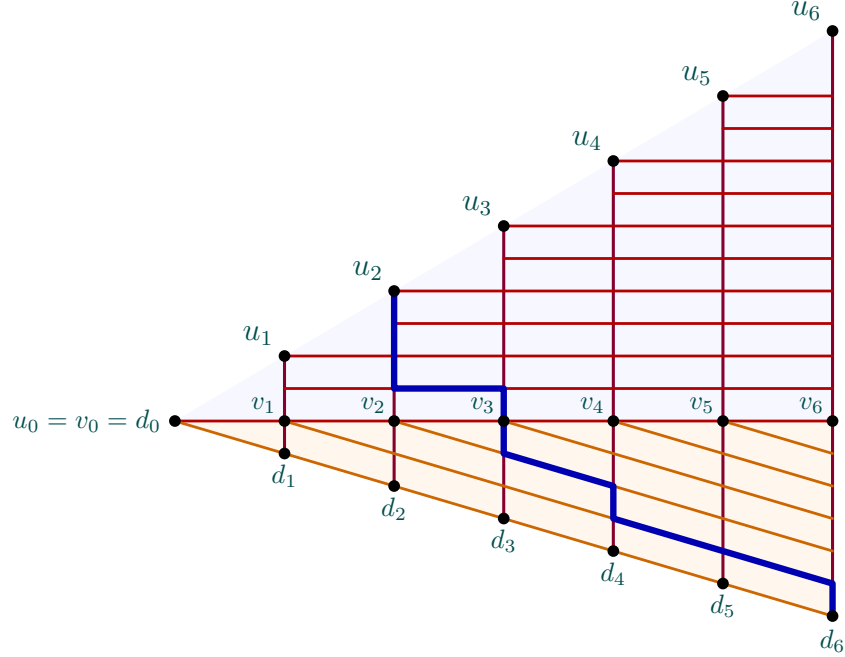
\begin{figure}
  \centering
  \begin{tikzpicture}[x=1.45cm,y=0.43cm]
    \def\N{6}
    \pgfmathtruncatemacro{\LastColumn}{\N-1}
    \fill[blue!3] (0,0) -- (\N,0) -- (\N,2*\N) -- cycle;
    \fill[orange!7] (0,0) -- (\N,-\N) -- (\N,0) -- cycle;

    \foreach \r in {1,...,\N} {
      \draw[awlane] (\r,-\r) -- (\r,2*\r);
    }
    \foreach \r in {0,...,\LastColumn} {
      \pgfmathtruncatemacro{\Smax}{2*\r}
      \foreach \s in {0,...,\Smax} {
        \draw[awstep] (\r+1,\s) -- (\r,\s);
      }
    }
    \foreach \r in {0,...,\LastColumn} {
      \foreach \h in {0,...,\r} {
        \draw[draw=orange!80!black,line width=1.05pt,line cap=round]
          (\r+1,\h-\r-1) -- (\r,\h-\r);
      }
    }

    \draw[draw=blue!70!black,line width=2.4pt,line cap=round,line join=round]
      (2,4) -- (2,1) -- (3,1) -- (3,0) -- (3,-1)
      -- (4,-2) -- (4,-3) -- (5,-4) -- (6,-5) -- (6,-6);

    \fill (0,0) circle (2.2pt)
      node[awendpoint,left=3pt,font=\small] {\(u_0=v_0=d_0\)};
    \foreach \i in {1,...,\N} {
      \fill (\i,0) circle (2.2pt)
        node[awendpoint,above left=1pt,font=\small] {\(v_{\i}\)};
      \fill (\i,2*\i) circle (2.2pt)
        node[awendpoint,above left=2pt] {\(u_{\i}\)};
    }
    \foreach \i in {1,...,\LastColumn} {
      \fill (\i,-\i) circle (2.2pt)
        node[awendpoint,below=2pt,font=\small] {\(d_{\i}\)};
    }
    \fill (\N,-\N) circle (2.2pt)
      node[awendpoint,below=3pt,font=\small] {\(d_6\)};

  \end{tikzpicture}
  \caption{The graph \(\mathcal{H}_6\), extending \(\mathcal{G}_6\)
    below the \( x \)-axis with south and southeast edges. The
    highlighted path starts at \(u_2\), enters the lower lattice from
    \(v_3\), and ends at \(d_6\); its weight is
    \(w_6^Z(2,1)\tau_6^Z(3)\tau_6^Z(4)\tau_6^Z(5)q^4\).}
  \label{fig:combined-lattice}
\end{figure}

\begin{corollary}\label{cor:combined-path-model}
  For \(0\le k\le n\), we have
  \begin{equation*}
    \widetilde{\nu}_{n,k}^Z
      =\sum_{\eta\in\mathcal{P}(u_k\to d_n)}\wt(\eta).
  \end{equation*}
\end{corollary}
\begin{proof}
  For \(0\le i\le n\), let \(\mathcal{P}_-(v_i\to d_n)\) denote the
  set of paths in \(\mathcal{P}(v_i\to d_n)\) that stay strictly below
  the \( x \)-axis after \(v_i\). Every path from \(u_k\) to \(d_n\)
  has a unique last vertex \(v_i\) on the \(x\)-axis. Splitting the
  path at this vertex gives
  \[
    \sum_{\eta\in\mathcal{P}(u_k\to d_n)}\wt(\eta)
    =\sum_{i=k}^{n} \sum_{\pi\in\mathcal{P}(u_k\to v_i)}\wt(\pi)
        \sum_{\rho\in\mathcal{P}_-(v_i\to d_n)}\wt(\rho).
  \]
  Since
  \[
    \sum_{\pi\in\mathcal{P}(u_k\to v_i)}\wt(\pi) = G^{Z}_n(i,k),
  \]
  by \Cref{thm:ASEP-positivity},
  it suffices to show that
  \begin{equation}\label{eq:8}
    T^Z_{n,i} = \sum_{\rho\in\mathcal{P}_-(v_i\to d_n)}\wt(\rho).
  \end{equation}

  To prove \eqref{eq:8}, note that every path
  \(\rho\in\mathcal{P}_-(v_i\to d_n)\) has \(i\) south steps and
  \(n-i\) southeast steps. Let \(h_j\) be the number of south steps
  following its \(j\)th southeast step. This gives a bijection with
  the sequences \(i\ge h_1\ge\cdots\ge h_{n-i}\ge0\). The \(j\)th
  southeast step starts at a vertex \((r,s)\) with \(r=i+j-1\) and
  \(s=-(i-h_j)-(j-1)\), so \(r+s=h_j\) and its weight is
  \(\tau_n^Z(i+j-1)q^{h_j}\). Hence
  \[
    \wt(\rho)=\prod_{r=i}^{n-1}\tau_n^Z(r)
      q^{h_1+\cdots+h_{n-i}}.
  \]
  Therefore,
  \[
    \sum_{\rho\in\mathcal{P}_-(v_i\to d_n)}\wt(\rho)
    =\prod_{r=i}^{n-1}\tau_n^Z(r)
      \sum_{i\ge h_1\ge\cdots\ge h_{n-i}\ge0}
        q^{h_1+\cdots+h_{n-i}}
    =\qbinom{n}{i}\prod_{r=i}^{n-1}\tau_n^Z(r)
    =T^Z_{n,i},
  \]
  which gives \eqref{eq:8}.
\end{proof}

\subsection{Reformulation in parameters \( a, b, c, d \)}\label{sec:proof-strategy}

In this subsection, we reformulate \Cref{thm:ASEP-positivity} using
the parameters \(a,b,c,d\). To this end, we introduce \(w_n(r,s)\),
\(G_n(i,k)\), and \(T_{n,k}\), which serve as the counterparts of
\(w_n^Z(r,s)\), \(G^{Z}_n(i,k)\), and \(T^Z_{n,k}\), respectively.

For fixed \(n\), define
\begin{equation}\label{eq:coordinate-weight-def}
  w_n(r,s)=
  \begin{cases}
    (1-\xi bdq^{n-r+2t-1})(1+aq^{r-t})(1+cq^{r-t}),
      & \text{if \(s=2t\)},\\
    (1-bdq^{n-r+2t})(\xi-acq^{r-t-1})(1-q^{r-t}),
      & \text{if \(s=2t+1\)}.
  \end{cases}
\end{equation}
Here \(0\le r<n\) and \(0\le s\le 2r\).
We extend this definition by setting \(w_n(r,s)=0\) for all integer
pairs \((r,s)\) outside this range.
Use the same graph \(\mathcal{G}_n\), with east edge weight
\(w_n(r,s)\) on \((r,s)\to(r+1,s)\) and south edge weight \(1\).

For \(0\le k,i\le n\), define
\begin{equation}\label{eq:Gn-def}
  G_n(i,k):=\sum_{\pi\in\mathcal{P}(u_k \to v_i)}\wt_n(\pi),
\end{equation}
where \(\wt_n(\pi)\) is the product of the edge weights of \(\pi\). By
definition, \(G_n(k,k)=1\) for \(0\le k\le n\) and \(G_n(i,k)=0\) if
\(0\le i<k\le n\).
Throughout the paper, we use the convention \(G_m(-1,-1)=1\), while
\(G_m(i,k)=0\) for every other pair \((i,k)\) outside the defining
range \(0\le i,k\le m\). In particular, \(G_m(j,-1)=0\) for
\(j\ne-1\).

For \(0\le k\le n\), define
\begin{equation}\label{eq:T-def}
T_{n,k}
= \xi^{n-k}\qbinom{n}{k}\qp{acq^k/\xi,-b,-d}{n-k}.
\end{equation}

Direct substitution using \Cref{def:parameter-change} yields the relation
\(w_n(r,s)=\frac{1-q}{\alpha\beta}w_n^Z(r,s)\), which implies
\begin{equation}\label{eq:G-scaling}
  G_n(i,k) = \left(\frac{1-q}{\alpha\beta}\right)^{i-k} G^{Z}_n(i,k),
\end{equation}
and similarly,
\begin{equation}\label{eq:T-scaling}
  T_{n,i} = \left(\frac{1-q}{\alpha\beta}\right)^{n-i} T^Z_{n,i}.
\end{equation}

For \(0\le k\le n\), define
\begin{equation}\label{eq:nuhat-def}
  \hat{\nu}_{n,k} = \hat{\nu}_{n,k}(\xi; a, b, c, d; q)
  := (q-1)^{n-k}\nu_{n, k}^Z\qp{abcdq^{n+k-1}}{n-k}.
\end{equation}
By \eqref{eq:abcd-ASEP} and \eqref{eq:nu_tilde}, we have
\begin{equation}\label{eq:nuhat-to-nu}
  \hat{\nu}_{n,k} = \left(\frac{1-q}{\alpha\beta}\right)^{n-k}\widetilde{\nu}_{n,k}^Z.
\end{equation}
Combining \eqref{eq:G-scaling}, \eqref{eq:T-scaling}, and \eqref{eq:nuhat-to-nu}, \Cref{thm:ASEP-positivity} is equivalent to the following identity in the parameters \( a, b, c, d \):

\begin{theorem}\label{thm:main}
  For \(0\le k\le n\), we have
  \begin{equation}\label{eq:main-identity}
    \hat{\nu}_{n,k} = \sum_{i=k}^{n} T_{n,i} G_n(i,k).
  \end{equation}
\end{theorem}

We prove this theorem in the next section. We will use the following
recurrence for \( G_n(i,k) \).

\begin{lemma}\label{lem:G-recurrence}
  For integers \(0\le k\le i\le n\), we have
  \begin{align} \label{eq:claim_1}
    G_{n+1}(i-1, k-1)
    &= G_n(i, k) - (w_n(i-1,0) + w_n(i-1,1))G_n(i-1, k)\\
    \notag
    &\quad + w_n(i-2,0)w_n(i-1,1)G_n(i-2, k).
  \end{align}
\end{lemma}

\begin{proof}
  The case \(i=0\) forces \(k=0\). By our convention and
  \(G_n(0,0)=1\), both sides are equal to \(1\). Now assume \(i\ge1\).
  The case \(k=0\) follows from the fact that every path from \(u_0\) is
  horizontal, which implies \(G_n(i,0)=w_n(i-1,0)G_n(i-1,0)\).

  Now suppose \(k\ge1\). Observe that by
  \eqref{eq:coordinate-weight-def}, we have
  \( w_{n+1}(r,s)=w_n(r+1,s+2) \) for \(0\le r<n-1\) and
  \(0\le s\le2r\). Since translating each path
  \( \pi: u_{k-1} \to v_{i-1} \) by \( (1,2) \) gives a path
  \( \sigma: u_k \to (i,2) \), we obtain
\[
  G_{n+1}(i-1, k-1) = \sum_{\pi: u_{k-1} \to v_{i-1}} \wt_{n+1}(\pi)
  = \sum_{\sigma: u_k \to (i,2)} \wt_n(\sigma).
\]
On the other hand, considering the last two steps of each path
\( \pi:u_k\to v_i \), we have
\begin{align*}
  G_{n}(i, k) 
  &= \sum_{\pi: u_k \to (i,2)} \wt_{n}(\pi) + (w_n(i-1,1)+w_n(i-1,0))
  \sum_{\pi: u_k \to (i-1,1)} \wt_{n}(\pi)\\
  & \quad + w_n(i-1,0)w_n(i-2,0) G_{n}(i-2, k).
\end{align*}
Similarly, considering the last step of each path
\( \pi:u_k\to v_{i-1} \), we have
\[
  G_{n}(i-1, k) 
  = \sum_{\pi: u_k \to (i-1,1)} \wt_{n}(\pi) + w_n(i-2,0) G_{n}(i-2, k).
\]
Combining the above equations, we obtain the desired recurrence.
\end{proof}

\section{Proof of the main theorem}\label{sec:proof-main}

In this section, we establish \Cref{thm:main} and therefore our main
result \Cref{thm:ASEP-positivity}. The strategy of the proof is as
follows. First, we give an explicit formula for the left-hand side
\( \hat{\nu}_{n,k} \). Then we convert the identity to an equivalent
form in which a double sum equals a product. After interchanging the
order of summation, we evaluate the inner sum and then the outer sum,
completing the proof.

\subsection{Proof of \Cref{thm:main}}

We first find an explicit formula for \( \hat{\nu}_{n,k} \). To do
this, we need the following result.

\begin{proposition} \cite[Proposition~4.21]{Corteel2026a}
  \label{prop:factorial_AW}
  We have
  \[
    \PAW_n(x; a, b, c, d | q)
    = \sum_{k=0}^n \rho_{n, k}(a,b,c,d;q) \prod_{r=0}^{k-1}\left(x - f_r(a,q)\right),
  \]
  where
  \begin{align*}
    \rho_{n, k}(a,b,c,d;q)
    &=
      (2a)^{k-n}q^{k(k-n)} \qbinom{n}{k}
      \frac{\qp{abq^k,acq^k,adq^k}{n-k}}{\qp{abcdq^{n+k-1}}{n-k}},\\
    f_r(a,q)
    &= \frac{aq^r + a^{-1}q^{-r}}{2}.
  \end{align*}
\end{proposition}

The \emph{elementary symmetric polynomials} \(e_r(u_1,\ldots,u_m)\)
are given by
\begin{equation}\label{eq:e}
 \prod_{j=1}^{m}(t+u_j)
 = \sum_{r=0}^{m}e_{m-r}(u_1,\ldots,u_m)t^r.
\end{equation}
In particular, \(e_0(u_1,\ldots,u_m)=1\), and
\(e_r(u_1,\ldots,u_m)=0\) unless \( 0\le r\le m \). Let
\begin{align}\label{eq:C-def}
  C_{n,j}&:=(-\xi/a)^{n-j}q^{-j(n-j)}\qbinom{n}{j}
  \qp{abq^j,acq^j/\xi,adq^j}{n-j},\\
  \notag
  g_r&:=(1+aq^r)(1+\xi a^{-1}q^{-r}).
\end{align}

\begin{lemma}
  \label{lem:nuhat-explicit}
  For \(0\le k\le n\), we have
  \[
    \hat{\nu}_{n,k}
    =\sum_{j=k}^{n}C_{n,j}(abcdq^{n+k-1};q)_{j-k}
    e_{j-k}(g_0,g_1,\ldots,g_{j-1}).
  \]
\end{lemma}

\begin{proof}
  By \eqref{eq:PZ-relation} and \Cref{prop:factorial_AW}, we have
  \begin{equation}\label{eq:2}
    P_n^Z(x)=
    \left(\frac{2\sqrt{\xi}}{1-q}\right)^n
    \sum_{j=0}^n
    \rho_{n,j}\left(\frac a{\sqrt{\xi}},b\sqrt{\xi},
      \frac c{\sqrt{\xi}},d\sqrt{\xi};q\right)
    \prod_{r=0}^{j-1}
    \left(y-f_r\left(\frac a{\sqrt{\xi}},q\right)\right),
  \end{equation}
  where
  \[
    y = \frac{(1-q)x - (1 +\xi)}{2\sqrt{\xi}}.
  \]
  Since
  \[
    y-f_r\left(\frac a{\sqrt{\xi}},q\right)
    =\frac{(1-q)x-(1+\xi)}{2\sqrt{\xi}}
    -\frac{aq^r+\xi a^{-1}q^{-r}}{2\sqrt{\xi}}
    =\frac{1-q}{2\sqrt{\xi}}
      \left(x-\frac{g_r}{1-q}\right),
  \]
  by \eqref{eq:e}, we have
  \[
    \prod_{r=0}^{j-1}
      \left(y-f_r\left(\frac a{\sqrt{\xi}},q\right)\right)
    =\frac{1}{(2\sqrt{\xi})^j}
      \sum_{\ell=0}^{j}(-1)^{j-\ell}(1-q)^\ell
      e_{j-\ell}(g_0,g_1,\ldots,g_{j-1})x^\ell.
  \]
  Then, by \eqref{eq:2} and \Cref{prop:factorial_AW}, the coefficient
  of \(x^k\) in \( P_n^Z(x) \) is equal to
  \[
    \nu_{n,k}^Z
    =(q-1)^{k-n}\sum_{j=k}^{n}
      \frac{C_{n,j}}
      {(abcdq^{n+j-1};q)_{n-j}}
      e_{j-k}(g_0,g_1,\ldots,g_{j-1}).
  \]
  By \eqref{eq:nuhat-def}, multiplying this identity by
  \((q-1)^{n-k}(abcdq^{n+k-1};q)_{n-k}\) gives
  the desired equation.
\end{proof}

Recall that the identity in \Cref{thm:main} is
\begin{equation}\label{eq:3}
  \hat{\nu}_{n,k} = \sum_{i=k}^{n} T_{n,i} G_n(i,k).
\end{equation}
Note that both sides of the identity are sums. To put the identity in a
more manageable form, we convert it into one in which a double sum
equals a product. Observe that \( T_{n,i} \) is a simple product and
the matrix \( (G_n(i,k)) \) can be inverted.

For a fixed degree \(n\), the identity \eqref{eq:3} can be restated in
the following matrix form:
\begin{equation*}
    \begin{pmatrix}
        \hat{\nu}_{n,0} & \cdots & \hat{\nu}_{n,n}
    \end{pmatrix}
    =
    \begin{pmatrix}
        T_{n,0} & \cdots & T_{n,n}
    \end{pmatrix}
    \begin{pmatrix}
        G_n(0,0) & \cdots & 0        \\
        \vdots   & \ddots & \vdots   \\
        G_n(n,0) & \cdots & G_n(n,n)
    \end{pmatrix}.
\end{equation*}
Since \(G_n(i,i) = 1\) and \(G_n(i,k) = 0\) for \(i < k\), the
displayed matrix is lower unitriangular, hence invertible. Let
\(\bigl(E_n(i,j)\bigr)_{i,j=0}^n\) be its inverse. Consequently, the
identity is equivalent to
\begin{equation*}
    \begin{pmatrix}
        T_{n,0} & \cdots & T_{n,n}
    \end{pmatrix}
    =
    \begin{pmatrix}
        \hat{\nu}_{n,0} & \cdots & \hat{\nu}_{n,n}
    \end{pmatrix}
    \begin{pmatrix}
        E_n(0,0) & \cdots & 0        \\
        \vdots   & \ddots & \vdots   \\
        E_n(n,0) & \cdots & E_n(n,n)
    \end{pmatrix}.
\end{equation*}
Therefore, \eqref{eq:3} is equivalent to
\begin{equation}\label{eq:main-inverse-entry}
    T_{n,k} = \sum_{i=k}^{n} \hat{\nu}_{n,i} E_n(i,k).
\end{equation}

Substituting the formula in \Cref{lem:nuhat-explicit} into
\eqref{eq:main-inverse-entry} gives
\begin{align}
  \notag
  T_{n,k}
  &=\sum_{i=k}^{n}E_n(i,k)\sum_{j=i}^{n}C_{n,j}
    (abcdq^{n+i-1};q)_{j-i}
    e_{j-i}(g_0,g_1,\ldots,g_{j-1})\\
  \label{eq:4}
  &=\sum_{j=k}^{n}C_{n,j}
    \sum_{i=k}^{j}(abcdq^{n+i-1};q)_{j-i}
    e_{j-i}(g_0,g_1,\ldots,g_{j-1})  E_n(i,k) .
\end{align}
The following lemma shows that the inner sum over \(i\) collapses into
an explicit product.

\begin{lemma}\label{lem:inner-sum}
For \(0\le k\le j\le n\), we have
\begin{equation}\label{eq:middle-product}
    \sum_{i=k}^{j} (abcdq^{n+i-1}; q)_{j-i} e_{j-i}(g_0, g_1, \dots, g_{j-1}) E_n(i,k) = B_n(j,k),
\end{equation}
where
\begin{equation}\label{eq:B-def}
  B_n(j,k)=
  \displaystyle
    (\xi/a)^{j-k}q^{\binom{k}{2}-\binom{j}{2}}\qbinom{j}{k}
    \qp{-abdq^{n-1},-aq^k,acq^k/\xi}{j-k}.
\end{equation}
\end{lemma}

The proof of \Cref{lem:inner-sum} is postponed to the next subsection. By \eqref{eq:4} and \Cref{lem:inner-sum}, our desired identity \eqref{eq:main-inverse-entry} follows from the following lemma.

\begin{lemma}\label{lem:outer-sum}
For \(0\le k\le n\), we have
\begin{equation*}
  \sum_{j=k}^{n}C_{n,j}B_n(j,k)=T_{n,k}.
\end{equation*}
\end{lemma}

\begin{proof}
  By \eqref{eq:C-def} and \eqref{eq:B-def}, the left-hand side is
\begin{align*}
  L&:=\sum_{j=k}^{n}(-\xi/a)^{n-j}(\xi/a)^{j-k}
    q^{-j(n-j)+\binom{k}{2}-\binom{j}{2}}
    \qbinom{n}{j}\qbinom{j}{k}\\
  &\quad\times\qp{abq^j,acq^j/\xi,adq^j}{n-j}
    \qp{-abdq^{n-1},-aq^k,acq^k/\xi}{j-k}\\
  &=\sum_{j=k}^{n} (-1)^{n-j} (\xi/a)^{n-k}
    q^{-j(n-j)+\binom{k}{2}-\binom{j}{2}}
    \qbinom{n}{k}\qbinom{n-k}{j-k} \\
  &\quad\times \qp{acq^k/\xi}{n-k} \qp{abq^j,adq^j}{n-j}
    \qp{-abdq^{n-1},-aq^k}{j-k}.
\end{align*}
Reindexing the sum by replacing \(j\) with \(j+k\) gives
\begin{equation}\label{eq:Lm-def}
  L= (-\xi/a)^{n-k} q^{-k(n-k)} \qp{acq^k/\xi}{n-k} \qbinom{n}{k} S,
\end{equation}
where
\[
  S = \sum_{j=0}^{n-k} (-1)^{j}
    q^{\binom{j+1}{2} - j(n-k)}
    \qbinom{n-k}{j}
    \qp{abq^{j+k},adq^{j+k}}{n-k-j}
    \qp{-abdq^{n-1},-aq^k}{j}.
\]
By the identities
\[
  \qbinom{m}{j} = \frac{(-1)^j q^{mj-\binom{j}{2}} (q^{-m};q)_j}{(q;q)_j},
  \qquad 
  \qp{xq^{j}}{m-j}=\frac{\qp{x}{m}}{\qp{x}{j}},
\]
and the $q$-Pfaff--Saalsch\"{u}tz summation
\cite[(II.12)]{GR}, we obtain
\begin{align*}
  S &= (abq^k;q)_{n-k} (adq^k;q)_{n-k} \cdot
    {}_3\phi_2\!\left[
    \begin{matrix}
      q^{k-n},\ -aq^k,\ -abdq^{n-1}\\
      abq^k,\ adq^k
    \end{matrix};q,q\right]\\
  &= (abq^k;q)_{n-k} (adq^k;q)_{n-k} \cdot
  \frac{\qp{-b,-q^{1-n+k}/d}{n-k}}
  {\qp{abq^k,q^{1-n}/ad}{n-k}}.
\end{align*}
Applying
\[
  \qp{x}{n-k}=(-x)^{n-k}q^{\binom{n-k}{2}}\qp{q^{1-n+k}/x}{n-k}
\]
successively with \(x=adq^k\) and \(x=-d\) gives
\[
  S =(-a)^{n-k}q^{k(n-k)}\qp{-b,-d}{n-k}.
\]
Substituting this into \eqref{eq:Lm-def}, we obtain
\[
  L=\xi^{n-k}\qbinom{n}{k}\qp{acq^k/\xi,-b,-d}{n-k}=T_{n,k},
\]
as desired.
\end{proof}

To complete the proof of \Cref{thm:main}, it remains to prove
\Cref{lem:inner-sum}. This will be done in the next subsection.

\subsection{Proof of \Cref{lem:inner-sum}}

For fixed \(0\le j\le n\), collecting the identities in
\Cref{lem:inner-sum} for \(0\le k\le j\) gives the following matrix
identity, where the indexed tuples are row vectors:
\[
  ((abcdq^{n+i-1}; q)_{j-i} e_{j-i}(g_0, g_1, \dots, g_{j-1}))_{i=0}^j
  (E_n(i,\ell))_{i,\ell=0}^j = (B_n(j,i))_{i=0}^j.
\]
The full matrices \((E_n(i,\ell))_{i,\ell=0}^n\) and
\((G_n(i,\ell))_{i,\ell=0}^n\) are lower triangular and inverse to each
other, so their leading \((j+1)\times(j+1)\) principal submatrices are
also inverses. Multiplying on the right by
\((G_n(i,\ell))_{i,\ell=0}^j\), we obtain the equivalent identity
\[
  ((abcdq^{n+i-1}; q)_{j-i} e_{j-i}(g_0, g_1, \dots, g_{j-1}))_{i=0}^j
   = (B_n(j,i))_{i=0}^j (G_n(i,\ell))_{i,\ell=0}^j.
\]
Taking the \(k\)-th component of this identity shows that
\Cref{lem:inner-sum} is equivalent to the following proposition.

\begin{proposition}\label{prop:path-decomposition}
For \(0\le k\le j\le n\), we have
\[
  (abcdq^{n+k-1}; q)_{j-k} e_{j-k}(g_0, g_1, \ldots, g_{j-1})
  = \sum_{i=k}^{j} B_n(j,i) G_n(i,k).
\]
\end{proposition}

To prove this proposition, we need a lemma. Recall that
\begin{align*}
  B_n(j,i)
  &= (\xi/a)^{j-i}q^{\binom{i}{2}-\binom{j}{2}}\qbinom{j}{i}
    \qp{-abdq^{n-1},-aq^i,acq^i/\xi}{j-i},\\
  w_n(i,0)&=(1-\xi bdq^{n-i-1})(1+aq^i)(1+cq^i),\\
  w_n(i,1)&=(1-bdq^{n-i})(\xi-acq^{i-1})(1-q^i),\\
  g_j&=(1+aq^j)(1+\xi a^{-1}q^{-j}).
\end{align*}
Throughout this section, we use the convention \(B_n(-1,-1)=1\) and
\(B_n(j,k)=0\) for all other pairs \((j,k)\) outside the range
\(0\le k\le j\).

\begin{lemma}\label{lem: Bn recurrence}
  For \(0\le i\le j\le n\), we have
    \begin{multline}
  B_n(j,i) = B_{n+1}(j-1,i-1)
  - \left( w_n(i,0)+w_n(i,1) \right)
  B_{n+1}(j-1,i)\\
  +w_n(i,0)w_n(i+1,1)
  B_{n+1}(j-1,i+1)
  +(1-abcdq^{j+n-2})g_{j-1}B_n(j-1,i).
\end{multline}
\end{lemma}

\begin{proof}
  The case \(i=j\) is immediate, since both sides equal \(1\); for
  \(j=0\), this uses the convention \(B_{n+1}(-1,-1)=1\). Suppose
  \(0\le i<j\). We claim that
\begin{equation}\label{eq:B-pair}
  \frac{1+aq^{j-1}}{1+aq^\ell}B_n(j-1,\ell)
  =B_{n+1}(j-1,\ell)-w_n(\ell+1,1)B_{n+1}(j-1,\ell+1)
\end{equation}
for \(-1\le \ell<j\), and that
\begin{multline}\label{eq:B-step}
  B_n(j,i)=\frac{1+aq^{j-1}}{1+aq^{i-1}}B_n(j-1,i-1)\\
  -\frac{1+aq^{j-1}}{1+aq^i}w_n(i,0)B_n(j-1,i)
  +(1-abcdq^{j+n-2})g_{j-1}B_n(j-1,i).
\end{multline}
Substituting \eqref{eq:B-pair} into \eqref{eq:B-step}, first with
\(\ell=i-1\) and then with \(\ell=i\), yields the stated recurrence. Thus it
suffices to prove these identities. To do this, we use the following
ratios, obtained directly from \eqref{eq:B-def}:
\begin{align}
  \frac{B_{n+1}(j-1,\ell)}{B_n(j-1,\ell)}
  &=\frac{1+abdq^{j+n-\ell-2}}{1+abdq^{n-1}},
  \label{eq:B-ratio-n}\\
  \frac{B_n(j,\ell)}{B_n(j-1,\ell)}
  &=\frac{(1-q^j)(1+abdq^{j+n-\ell-2})(1+aq^{j-1})(\xi-acq^{j-1})}
    {aq^{j-1}(1-q^{j-\ell})},
  \label{eq:B-ratio-m}\\
  \frac{B_n(j-1,\ell+1)}{B_n(j-1,\ell)}
  &=\frac{aq^\ell(1-q^{j-\ell-1})}
    {(1-q^{\ell+1})(1+abdq^{j+n-\ell-3})(1+aq^\ell)(\xi-acq^\ell)}.
  \label{eq:B-ratio-i}
\end{align}
The first two hold for \(0\le \ell<j\), and the third for \(0\le \ell<j-1\).

We first prove \eqref{eq:B-pair}. For \(0\le \ell<j-1\), combine
\eqref{eq:B-ratio-n} with \eqref{eq:B-ratio-i} at \(n+1\) to express $B_{n+1}(j-1,\ell+1)/B_n(j-1,\ell)$.
Using \eqref{eq:coordinate-weight-def}
and dividing by \(B_n(j-1,\ell)\), we reduce \eqref{eq:B-pair} to
\[
  \frac{1+aq^{j-1}}{1+aq^\ell}
  =\frac{(1+abdq^{j+n-\ell-2})(1+aq^\ell)-(aq^\ell-abdq^{n-1})(1-q^{j-\ell-1})}
    {(1+abdq^{n-1})(1+aq^\ell)}.
\]
The case \(\ell=j-1\) is immediate. The identity also holds for \(\ell=-1\),
since \(w_n(0,1)=0\).

We next prove \eqref{eq:B-step} by dividing both sides by
\((1+aq^{j-1})B_n(j-1,i)\). Apply \eqref{eq:B-ratio-m} at \(\ell=i\)
and, when \(i>0\), take the reciprocal of \eqref{eq:B-ratio-i} at \(\ell=i-1\).
For \(i=0\), the term containing \(B_n(j-1,-1)\) is zero.
In either case, we obtain
\begin{align*}
  &\frac{B_n(j,i)}{(1+aq^{j-1})B_n(j-1,i)}
    -\frac{B_n(j-1,i-1)}{(1+aq^{i-1})B_n(j-1,i)}\\
  &=\frac{1+abdq^{j+n-i-2}}{1-q^{j-i}}
    \left[(1-q^j)\left(\frac{\xi}{aq^{j-1}}-c\right)
      -(1-q^i)\left(\frac{\xi}{aq^{i-1}}-c\right)\right]\\
  &=(1+abdq^{j+n-i-2})\left(\frac{\xi}{aq^{j-1}}-cq^i\right)\\
  &=(1-abcdq^{j+n-2})\left(1+\frac{\xi}{aq^{j-1}}\right)
    -(1-\xi bdq^{n-i-1})(1+cq^i)\\
  &=\frac{(1-abcdq^{j+n-2})g_{j-1}}{1+aq^{j-1}}
    -\frac{w_n(i,0)}{1+aq^i}.
\end{align*}
This proves \eqref{eq:B-step}.
\end{proof}

We are now ready to prove \Cref{prop:path-decomposition}.

\begin{proof}[Proof of \Cref{prop:path-decomposition}]
Let \(L^{(n)}_{j,k}\) and \(R^{(n)}_{j,k}\) be the left-hand side and
right-hand side of the identity in \Cref{prop:path-decomposition}:
\begin{align*}
 L^{(n)}_{j,k} &:= (abcdq^{n+k-1}; q)_{j-k} e_{j-k}(g_0, g_1, \ldots, g_{j-1}),\\
  R^{(n)}_{j,k} &:= \sum_{i=k}^{j} B_n(j,i) G_n(i,k).
\end{align*}
For \(0\le k<j\le n\), we claim that they satisfy the same recurrences:
\begin{align}
  L^{(n)}_{j,k}
  &=L^{(n+1)}_{j-1,k-1}
  +(1-abcdq^{j+n-2})g_{j-1}L^{(n)}_{j-1,k},
  \label{eq:L-contiguous}\\
  R^{(n)}_{j,k}
  &=R^{(n+1)}_{j-1,k-1}
  +(1-abcdq^{j+n-2})g_{j-1}R^{(n)}_{j-1,k},
  \label{eq:R-contiguous}
\end{align}
where we define \(L^{(n)}_{m,-1}=R^{(n)}_{m,-1}=0\) for all \( m \).
Together with the diagonal values \(L^{(n)}_{j,j}=R^{(n)}_{j,j}=1\),
these recurrences prove the proposition by induction on \(j\).

We first prove \eqref{eq:L-contiguous}.
Since
\[
  e_{j-k}(g_0, \ldots, g_{j-1}) = g_{j-1} e_{j-k-1}(g_0, \ldots, g_{j-2}) + e_{j-k}(g_0, \ldots, g_{j-2}),
\]
we have
\begin{align*}
  L^{(n)}_{j,k}
  &=(abcdq^{n+k-1};q)_{j-k} \bigl(
    g_{j-1}e_{j-k-1}(g_0,\ldots,g_{j-2})
    +e_{j-k}(g_0,\ldots,g_{j-2})\bigr)\\
  &=(1-abcdq^{j+n-2})g_{j-1}L^{(n)}_{j-1,k}  +L^{(n+1)}_{j-1,k-1}.
\end{align*}
For \(k=0\), the term \(e_j(g_0,\ldots,g_{j-2})\) vanishes, in
agreement with the convention \(L^{(n+1)}_{j-1,-1}=0\).
This proves \eqref{eq:L-contiguous} for all \(0\le k<j\).

We next prove \eqref{eq:R-contiguous}. Applying
\Cref{lem:G-recurrence}, we obtain
\[
  R^{(n+1)}_{j-1,k-1}
  = \sum_{i=k}^{j} B_{n+1}(j-1,i-1) G_{n+1}(i-1, k-1) = X-Y+Z,
\]
where
\begin{align*}
  X &= \sum_{i=k}^{j} B_{n+1}(j-1,i-1) G_n(i,k),\\
  Y &= \sum_{i=k}^{j} B_{n+1}(j-1,i-1)(w_n(i-1,0) + w_n(i-1,1)) G_n(i-1,k), \\
  Z &= \sum_{i=k}^{j} B_{n+1}(j-1,i-1) w_n(i-2,0)w_n(i-1,1) G_n(i-2,k).
\end{align*}
Since \(B_{n+1}(j-1,\ell)=0\) for \(\ell>j-1\) and
\(G_n(\ell,k)=0\) for \(\ell<k\), shifting the summation index
\(i\mapsto i+1\) in \(Y\) and \(i\mapsto i+2\) in \(Z\) gives
\begin{align*}
  Y &= \sum_{i=k}^{j} B_{n+1}(j-1,i)(w_n(i,0) + w_n(i,1)) G_n(i,k), \\
  Z &= \sum_{i=k}^{j} B_{n+1}(j-1,i+1)w_n(i,0)w_n(i+1,1) G_n(i,k).
\end{align*}
Combining these expressions with the definition of \(X\), we obtain
\begin{multline*}
  X-Y+Z
  = \sum_{i=k}^{j} \Bigl( B_{n+1}(j-1,i-1)
     - B_{n+1}(j-1,i)(w_n(i,0) + w_n(i,1))\\
   + B_{n+1}(j-1,i+1)
  w_n(i,0)w_n(i+1,1) \Bigr) G_n(i,k).
\end{multline*}
Then, by \Cref{lem: Bn recurrence}, we have
\begin{align*}
 R^{(n+1)}_{j-1,k-1}=X-Y+Z
  &= \sum_{i=k}^{j} \Bigl( B_n(j,i) - (1 - abcdq^{j+n-2}) g_{j-1} B_n(j-1,i) \Bigr) G_n(i,k) \\
  &= R^{(n)}_{j,k} - (1 - abcdq^{j+n-2}) g_{j-1} R^{(n)}_{j-1,k}.
\end{align*}
This proves \eqref{eq:R-contiguous} for all \(0\le k<j\), which completes the proof.
\end{proof}

\appendix

\section{Heuristic discovery of the path formula for $\widetilde{\nu}^{Z}_{n,k}$}
\label[appendix]{sec:appendix}

In this appendix, we outline the heuristic reasoning and combinatorial
intuition that led us to the formula for $\widetilde{\nu}^{Z}_{n,k}$
in \Cref{thm:ASEP-positivity}. Our strategy had three steps. First, we
found an explicit formula for \( \widetilde{\nu}^{Z}_{n,k} \) when
\( k=0 \). Second, we expressed \( \widetilde{\nu}^{Z}_{n,k} \) as a
weight generating function for paths under a certain specialization of
the parameters. Third, we modified the weights so that
\( \widetilde{\nu}^{Z}_{n,k} \), in the general case, is the
generating function for the same paths.

Our first observation was the following formula for
\( \widetilde{\nu}^{Z}_{n,k} \) when \( k=0 \).

\begin{proposition}\label{prop:constant-term}
For every integer $n\ge0$, we have
\[
  \widetilde{\nu}^{Z}_{n,0}
  =\sum_{k=0}^{n}\qbinom{n}{k}
    \prod_{i=0}^{k-1}
      (\beta+\xi\delta q^{n-1-i})
      \bigl((\alpha+\gamma q^i)[i]_q+q^i\bigr)
    \prod_{i=0}^{n-k-1}
      \bigl((\beta+\delta q^i)[i]_q+q^i\bigr)
      (\xi\alpha+\gamma q^{n-1-i}).
\]
\end{proposition}

\begin{proof}
  The following identity is known (see \cite[Equation~(15.2.8)]{Ismail}, rewritten here for the monic normalization $\PAW_n$):
\begin{align}\label{rr}
    \frac{2^n(abcdq^{n-1};q)_n \PAW_n(x;a,b,c,d\vert q)}{(q,ac,bd;q)_n}
    =
    \sum_{k=0}^{n}
    \frac{(az,cz;q)_k(bz^{-1},dz^{-1};q)_{n-k}}{(q,ac;q)_k(q,bd;q)_{n-k}}
    z^{(n-2k)},
\end{align}
where \( x= (z+z^{-1})/2 \).
By \eqref{eq:PZ-relation}, we have
\begin{equation*}
  \nu^{Z}_{n,0}
  =\PZ_n(0)
  =\left(\frac{2\sqrt{\xi}}{1-q}\right)^n
  \PAW_{n}\left(-\frac{1+\xi}{2\sqrt{\xi}};\frac{a}{\sqrt{\xi}},
  b\sqrt{\xi},\frac{c}{\sqrt{\xi}},d\sqrt{\xi} \bigg| q\right).
\end{equation*}
Therefore, substituting \(z=-\sqrt{\xi}\) into \eqref{rr} and using
the change of variables in \Cref{def:parameter-change}, we obtain the
formula.
\end{proof}

\Cref{prop:constant-term}, which corresponds to the special case $k=0$ of \Cref{thm:ASEP-positivity}, yields a manifestly positive formula for $\widetilde{\nu}^{Z}_{n,0}$. A particularly appealing feature of this proof is that the positive expression does not need to be guessed in advance; rather, it arises directly and mechanically from \eqref{rr}.

Now we proceed to find a path model for \( \widetilde{\nu}^{Z}_{n,k} \).
Since finding a path model directly for the general case is
challenging, we first consider a special case.

Let $\nu^{S}_{n,k}$ denote the specialization of
$\widetilde{\nu}^{Z}_{n,k}$ at
\begin{equation}\label{eq: specialization}
  \beta = \xi = 1 \quad \text{and} \quad \delta = 0.
\end{equation}
Let \(\mathcal{S}_n\) be the directed graph with vertex set \(V_S\)
and edge set \(E_S\) given by
\begin{align*}
  V_S&=\bigl\{(r,s)\in\mathbb{Z}^2:0\le r\le n,\ 0\le s\le 2r+1\bigr\},\\
  E_S&=\bigl\{(u,v)\in V_S^2:u=(r,s),\ v=(r,s-1)\bigr\}
     \cup\bigl\{(u,v)\in V_S^2:u=(r,s),\ v=(r+1,s)\bigr\}.
\end{align*}
Every south edge \((r,s)\to(r,s-1)\) has weight \(1\), and the east
edge \((r,s)\to(r+1,s)\) has weight
\[
  w^S(r,s):=
  \begin{cases}
    X_{r-t},&\text{if \(s=2t+1\)},\\
    Y_{r-t},&\text{if \(s=2t\)},
  \end{cases}
\]
where
\begin{align*}
  X_i &= (\alpha + \gamma q^{i})[i]_q + q^{i}, \\
  Y_i &= (\alpha + \gamma q^{i})[i+1]_q.
\end{align*}
For \(0\le i\le n\), let \( U_i=(i,2i+1) \) and \( V_i=(i,0) \). For
vertices \(u,v\) of \(\mathcal{S}_n\), let \(\mathcal{P}_S(u\to v)\)
denote the set of directed paths from \(u\) to \(v\), and let
\(\wt^S(\pi)\) be the product of the edge weights of a path \(\pi\).

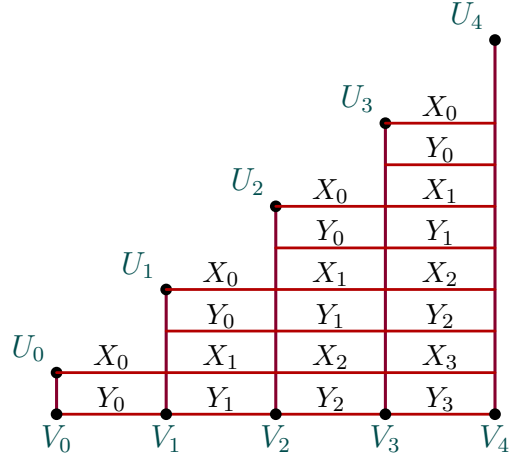
\begin{figure}
  \centering
  \begin{tikzpicture}[x=1.45cm,y=0.55cm]
  \def\N{4}
  \pgfmathtruncatemacro{\LastColumn}{\N-1}

  \foreach \k in {0,...,\N} {
    \pgfmathtruncatemacro{\Ysink}{2*\k}
      \draw[awlane] (\k,-1) -- (\k,\Ysink);
    \fill (\k,\Ysink) circle (2.2pt)
      node[awendpoint,above left=2pt] {\(U_{\k}\)};
    }

  \foreach \r in {0,...,\LastColumn} {
    \pgfmathtruncatemacro{\Smax}{2*\r}
    \foreach \s in {-1,...,\Smax} {
      \draw[awstep] (\r,\s) -- (\r+1,\s);
    }
  }
  
  \foreach \r in {0,...,\LastColumn} {
    \pgfmathtruncatemacro{\Smax}{\r}
    \foreach \s in {0,...,\Smax} {
      \pgfmathtruncatemacro{\ind}{\r - \s}
      \node[above,inner sep=1pt] at (\r + 0.5, 2*\s) {\( X_{\ind} \)};
    }
  }
  \foreach \r in {0,...,\LastColumn} {
    \pgfmathtruncatemacro{\Smax}{\r}
    \foreach \s in {0,...,\Smax} {
      \pgfmathtruncatemacro{\ind}{\r - \s}
      \node[above,inner sep=1pt] at (\r + 0.5, 2*\s - 1) {\( Y_{\ind} \)};
    }
  }
  
  \foreach \i in {0,...,\N} {
    \fill (\i,-1) circle (2.2pt)
      node[awendpoint,below=2pt] {\(V_{\i}\)};
  }
  \end{tikzpicture}
  \caption{A lattice model for \( \nu_{4,k}^S \).}
  \label{fig:lattice_initial}
\end{figure}

Our second observation was that
\[
  \nu^S_{n,k}
  =\sum_{\pi\in\mathcal{P}_S(U_k\to V_n)}\wt^S(\pi)
  \qquad(0\le k\le n).
\]
See \Cref{fig:lattice_initial} for \(\mathcal{S}_4\). Guided by this
observation, our strategy for finding a path model for
$\widetilde{\nu}^{Z}_{n,k}$ in the general case was to modify the
weights \( X_i \) and \( Y_i \) in \Cref{fig:lattice_initial} so that
\begin{equation}\label{eq:6}
  \widetilde{\nu}^Z_{n,k}
  =\sum_{\pi\in\mathcal{P}_S(U_k\to V_n)}\wt(\pi)
  \qquad(0\le k\le n),
\end{equation}
where \( \wt(\pi) \) is the product of the modified weights of the
steps in \( \pi \). Motivated by the factors appearing in the identity
in \Cref{prop:constant-term}, we found that generalizing $X_i$ and
$Y_i$ required introducing an additional index \( \ell \):
\begin{align*}
    X_{i, \ell} &=  \bigl((\alpha + \gamma q^{i}) [i]_q + q^{i}\bigr) (\beta + \xi \delta q^{\ell}), \\
    Y_{i, \ell} &= (\xi \alpha + \gamma q^{i}) (\beta + \delta q^{\ell}) [i+1]_q.
\end{align*}
Note that for every integer $\ell$, the new weights $X_{i,\ell}$ and
$Y_{i,\ell}$ specialize to $X_i$ and $Y_i$ under \eqref{eq:
  specialization}. In what follows, we refer to the second index
$\ell$ as the \emph{decoration index}.

Our goal is now to replace the weights \( X_i \) and \( Y_{i'} \) in
\Cref{fig:lattice_initial} with \( X_{i,\ell} \) and \( Y_{i',\ell'} \),
respectively, choosing \( \ell \) and \( \ell' \) so that
\eqref{eq:6} holds.

First, we considered the weights of the east steps in the last column.
We observed that \( \widetilde{\nu}^{Z}_{n,n-1} \) was very close to
\[
   (X_{0,2n-2}+Y_{0,2n-3})+(X_{1,2n-4}+Y_{1,2n-5})+\dots+ (X_{n-1,0}+Y_{n-1,-1}).
\]
Moreover, if we replace the last term \( Y_{n-1,-1} \)
by
\[
  Y_{n-1}^{(n)} := (\xi\alpha+\gamma q^{n-1})[n]_q,
\]
then we have the equality
\[
  \widetilde{\nu}^{Z}_{n,n-1} = (X_{0,2n-2}+Y_{0,2n-3})+(X_{1,2n-4}+Y_{1,2n-5})+\dots+ (X_{n-1,0}+Y^{(n)}_{n-1}).
\]
This determines the weights of the east steps in the last column
between \( x=n-1 \) and \( x=n \) as in the rightmost diagram in
\Cref{fig:first_lift_one_line}.

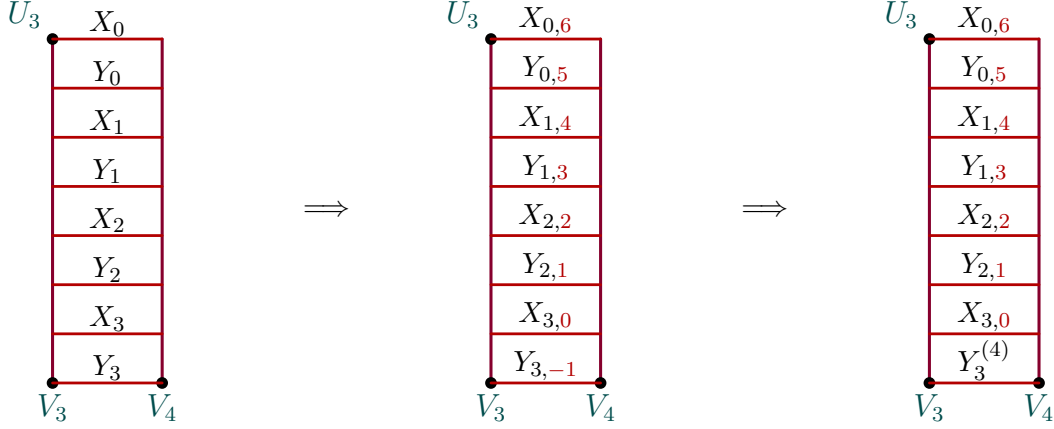
\begin{figure}
  \centering
  \begin{tikzpicture}[x=1.45cm,y=0.65cm]
    \begin{scope}
      \draw[awlane] (0,-1) -- (0,6);
      \draw[awlane] (1,-1) -- (1,6);
      \fill (0,6) circle (2.2pt)
        node[awendpoint,above left=2pt] {$U_3$};
      \fill (0,-1) circle (2.2pt)
        node[awendpoint,below=2pt] {$V_3$};
      \fill (1,-1) circle (2.2pt)
        node[awendpoint,below=2pt] {$V_4$};

      \foreach \s in {-1,...,6} {
        \draw[awstep] (0,\s) -- (1,\s);
      }
      \foreach \i in {0,...,3} {
        \pgfmathtruncatemacro{\Xheight}{6-2*\i}
        \pgfmathtruncatemacro{\Yheight}{5-2*\i}
        \node[above,inner sep=1pt] at (0.5,\Xheight) {$X_{\i}$};
        \node[above,inner sep=1pt] at (0.5,\Yheight) {$Y_{\i}$};
      }
    \end{scope}

    \begin{scope}[shift={(4,0)}]
      \draw[awlane] (0,-1) -- (0,6);
      \draw[awlane] (1,-1) -- (1,6);
      \fill (0,6) circle (2.2pt)
        node[awendpoint,above left=2pt] {$U_3$};
      \fill (0,-1) circle (2.2pt)
        node[awendpoint,below=2pt] {$V_3$};
      \fill (1,-1) circle (2.2pt)
        node[awendpoint,below=2pt] {$V_4$};

      \foreach \s in {-1,...,6} {
        \draw[awstep] (0,\s) -- (1,\s);
      }
      \foreach \i in {0,...,3} {
        \pgfmathtruncatemacro{\Xheight}{6-2*\i}
        \pgfmathtruncatemacro{\Yheight}{5-2*\i}
        \node[above,inner sep=1pt] at (0.5,\Xheight)
          {$X_{\i,\textcolor{red!70!black}{\Xheight}}$};
        \node[above,inner sep=1pt] at (0.5,\Yheight)
          {$Y_{\i,\textcolor{red!70!black}{\Yheight}}$};
      }
    \end{scope}

    \begin{scope}[shift={(8,0)}]
      \draw[awlane] (0,-1) -- (0,6);
      \draw[awlane] (1,-1) -- (1,6);
      \fill (0,6) circle (2.2pt)
        node[awendpoint,above left=2pt] {$U_3$};
      \fill (0,-1) circle (2.2pt)
        node[awendpoint,below=2pt] {$V_3$};
      \fill (1,-1) circle (2.2pt)
        node[awendpoint,below=2pt] {$V_4$};

      \foreach \s in {-1,...,6} {
        \draw[awstep] (0,\s) -- (1,\s);
      }
      \foreach \i in {0,...,3} {
        \pgfmathtruncatemacro{\Xheight}{6-2*\i}
        \node[above,inner sep=1pt] at (0.5,\Xheight)
          {$X_{\i,\textcolor{red!70!black}{\Xheight}}$};
      }
      \foreach \i in {0,...,2} {
        \pgfmathtruncatemacro{\Yheight}{5-2*\i}
        \node[above,inner sep=1pt] at (0.5,\Yheight)
          {$Y_{\i,\textcolor{red!70!black}{\Yheight}}$};
      }
      \node[above,inner sep=1pt] at (0.5,-1) {$Y_3^{(4)}$};
    \end{scope}

    \node at (2.5,2.5) {\( \Longrightarrow \)};
    \node at (6.5,2.5) {\( \Longrightarrow \)};
  \end{tikzpicture}
  \caption{Modification of the weights in the last column when \(n=4\).}
  \label{fig:first_lift_one_line}
\end{figure}

Next, we considered the second-to-last column between \( x=n-2 \) and
\( x=n-1 \). Observe that for the last column, the decoration indices
(represented with red letters) decrease by \( 1 \) as the height of an
east step decreases by \( 1 \). Motivated by this example, we expected
that the weights of the east steps from top to bottom in the
second-to-last column would be
\[
  X_{0,N}, Y_{0,N-1}, X_{1,N-2}, Y_{1,N-3},\dots, X_{n-2,N-2n+4}, Y^{(n)}_{n-2}
\]
for some choices of \( N \) and \( Y^{(n)}_{n-2} \). We found that the
most suitable choice of $N$ should be \( N=2n-3 \), which gave the
minimal number of terms in the difference
$\widetilde{\nu}^{Z}_{n,n-2}-\sum_{\pi\in \mathcal{P}_S(U_{n-2}\to V_n)}\wt(\pi)$.
Note that once \( N \) is fixed, the weight \( Y^{(n)}_{n-2} \) is
uniquely determined and is given by
\[
  Y^{(n)}_{n-2} = \frac{[n-1]_q}{[2]_q}
  (\xi\alpha+\gamma q^{n-2})(\beta+\delta q+q).
\]
The resulting modification of the weights in the last two columns is
shown in \Cref{fig:first_lift_two_columns}.

\begin{figure}
  \centering
  \begin{tikzpicture}[x=1.45cm,y=0.65cm]
    \begin{scope}
      \foreach \k in {2,...,4} {
        \ifnum\k<4
          \draw[awlane] (\k,-1) -- (\k,2*\k);
          \fill (\k,2*\k) circle (2.2pt)
            node[awendpoint,above left=2pt] {$U_{\k}$};
        \else
          \draw[awlane] (\k,-1) -- (\k,6);
        \fi
        \fill (\k,-1) circle (2.2pt)
          node[awendpoint,below=2pt] {$V_{\k}$};
      }
      \foreach \r in {2,3} {
        \pgfmathtruncatemacro{\TopHeight}{2*\r}
        \foreach \s in {-1,...,\TopHeight} {
          \draw[awstep] (\r,\s) -- (\r+1,\s);
        }
        \foreach \i in {0,...,\r} {
          \pgfmathtruncatemacro{\Xheight}{2*\r-2*\i}
          \pgfmathtruncatemacro{\Yheight}{\Xheight-1}
          \node[above,inner sep=1pt] at (\r+0.5,\Xheight) {$X_{\i}$};
          \node[above,inner sep=1pt] at (\r+0.5,\Yheight) {$Y_{\i}$};
        }
      }
    \end{scope}

    \begin{scope}[shift={(4,0)}]
      \foreach \k in {2,...,4} {
        \ifnum\k<4
          \draw[awlane] (\k,-1) -- (\k,2*\k);
          \fill (\k,2*\k) circle (2.2pt)
            node[awendpoint,above left=2pt] {$U_{\k}$};
        \else
          \draw[awlane] (\k,-1) -- (\k,6);
        \fi
        \fill (\k,-1) circle (2.2pt)
          node[awendpoint,below=2pt] {$V_{\k}$};
      }
      \foreach \r in {2,3} {
        \pgfmathtruncatemacro{\TopHeight}{2*\r}
        \foreach \s in {-1,...,\TopHeight} {
          \draw[awstep] (\r,\s) -- (\r+1,\s);
        }
      }
      \node[above,inner sep=1pt] at (2.5,4)
        {$X_{0,\textcolor{red!70!black}{N}}$};
      \node[above,inner sep=1pt] at (2.5,3)
        {$Y_{0,\textcolor{red!70!black}{N-1}}$};
      \node[above,inner sep=1pt] at (2.5,2)
        {$X_{1,\textcolor{red!70!black}{N-2}}$};
      \node[above,inner sep=1pt] at (2.5,1)
        {$Y_{1,\textcolor{red!70!black}{N-3}}$};
      \node[above,inner sep=1pt] at (2.5,0)
        {$X_{2,\textcolor{red!70!black}{N-4}}$};
      \node[above,inner sep=1pt] at (2.5,-1) {$Y_2^{(4)}$};
      \foreach \i in {0,...,3} {
        \pgfmathtruncatemacro{\Xheight}{6-2*\i}
        \node[above,inner sep=1pt] at (3.5,\Xheight)
          {$X_{\i,\textcolor{red!70!black}{\Xheight}}$};
      }
      \foreach \i in {0,...,2} {
        \pgfmathtruncatemacro{\Yheight}{5-2*\i}
        \node[above,inner sep=1pt] at (3.5,\Yheight)
          {$Y_{\i,\textcolor{red!70!black}{\Yheight}}$};
      }
      \node[above,inner sep=1pt] at (3.5,-1) {$Y_3^{(4)}$};
    \end{scope}

    \begin{scope}[shift={(8,0)}]
      \foreach \k in {2,...,4} {
        \ifnum\k<4
          \draw[awlane] (\k,-1) -- (\k,2*\k);
          \fill (\k,2*\k) circle (2.2pt)
            node[awendpoint,above left=2pt] {$U_{\k}$};
        \else
          \draw[awlane] (\k,-1) -- (\k,6);
        \fi
        \fill (\k,-1) circle (2.2pt)
          node[awendpoint,below=2pt] {$V_{\k}$};
      }
      \foreach \r in {2,3} {
        \pgfmathtruncatemacro{\TopHeight}{2*\r}
        \foreach \s in {-1,...,\TopHeight} {
          \draw[awstep] (\r,\s) -- (\r+1,\s);
        }
        \foreach \i in {0,...,\r} {
          \pgfmathtruncatemacro{\Xheight}{2*\r-2*\i}
          \pgfmathtruncatemacro{\Xdecoration}{\r+3-2*\i}
          \node[above,inner sep=1pt] at (\r+0.5,\Xheight)
            {$X_{\i,\textcolor{red!70!black}{\Xdecoration}}$};
        }
        \pgfmathtruncatemacro{\LastYindex}{\r-1}
        \foreach \i in {0,...,\LastYindex} {
          \pgfmathtruncatemacro{\Yheight}{2*\r-2*\i-1}
          \pgfmathtruncatemacro{\Ydecoration}{\r+2-2*\i}
          \node[above,inner sep=1pt] at (\r+0.5,\Yheight)
            {$Y_{\i,\textcolor{red!70!black}{\Ydecoration}}$};
        }
        \node[above,inner sep=1pt] at (\r+0.5,-1) {$Y_{\r}^{(4)}$};
      }
    \end{scope}

    \node at (5,2.5) {\( \Longrightarrow \)};
    \node at (9,2.5) {\( \Longrightarrow \)};
  \end{tikzpicture}
  \caption{Modification of the weights in the last two columns when \(n=4\).}
  \label{fig:first_lift_two_columns}
\end{figure}
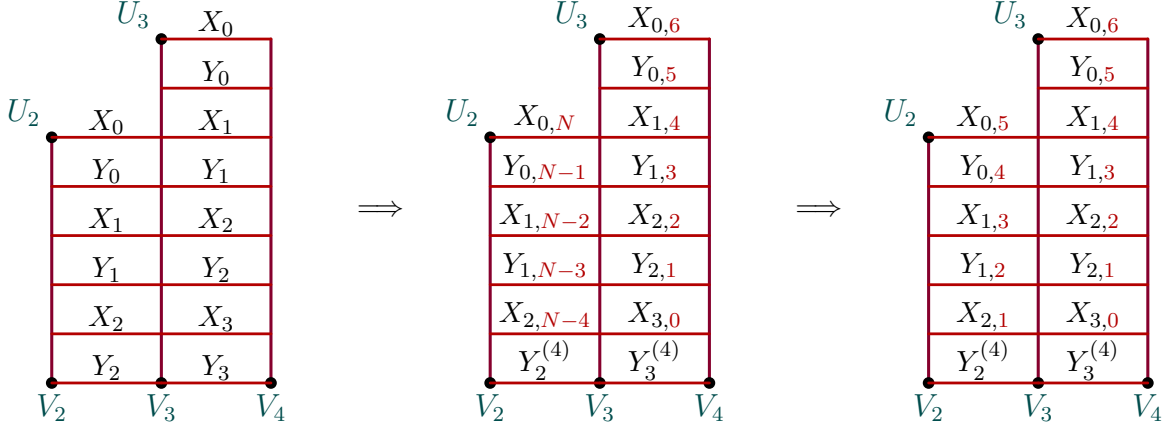

At this point, the decoration indices in the last two columns show a
clear pattern: they decrease by \( 1 \) from top to bottom and
from left to right. Computer experiments confirmed that this pattern continued
as shown in \Cref{fig:first_lift_2_vertical}.

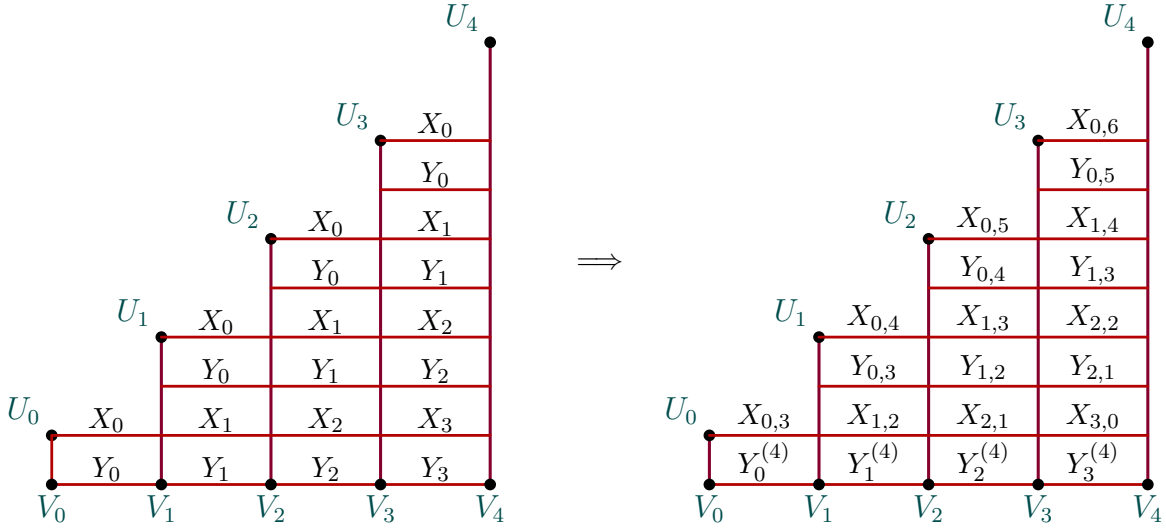
\begin{figure}
  \centering
  \begin{tikzpicture}[x=1.45cm,y=0.65cm]
    \def\N{4}
    \pgfmathtruncatemacro{\LastColumn}{\N-1}
    \begin{scope}
      \foreach \k in {0,...,\N} {
        \pgfmathtruncatemacro{\Ysink}{2*\k}
        \ifnum\k>0
          \draw[awlane] (\k,-1) -- (\k,\Ysink);
        \fi
        \fill (\k,\Ysink) circle (2.2pt)
        node[awendpoint,above left=2pt] {\(U_{\k}\)};
      }

      \foreach \r in {0,...,\LastColumn} {
        \pgfmathtruncatemacro{\Smax}{2*\r}
        \foreach \s in {0,...,\Smax} {
          \draw[awstep] (\r,\s) -- (\r+1,\s);
        }
      }
      \draw[awstep] (0, 0) -- (0, -1) -- (\N, -1);
      
      \foreach \r in {0,...,\LastColumn} {
        \pgfmathtruncatemacro{\Smax}{\r}
        \foreach \s in {0,...,\Smax} {
          \pgfmathtruncatemacro{\ind}{\r - \s}
          \node[above,inner sep=1pt] at (\r + 0.5, 2*\s) {\( X_{\ind} \)};
        }
      }
      \foreach \r in {0,...,\LastColumn} {
        \pgfmathtruncatemacro{\Smax}{\r}
        \foreach \s in {0,...,\Smax} {
          \pgfmathtruncatemacro{\ind}{\r - \s}
          \node[above,inner sep=1pt] at (\r + 0.5, 2*\s - 1) {\( Y_{\ind} \)};
        }
      }
      
      \foreach \i in {0,...,\N} {
        \fill (\i,-1) circle (2.2pt)
        node[awendpoint,below=2pt] {\(V_{\i}\)};
      }
    \end{scope}
    \begin{scope}[shift = {(6, 0)}]
      \foreach \k in {0,...,\N} {
        \pgfmathtruncatemacro{\Ysink}{2*\k}
        \draw[awlane] (\k,-1) -- (\k,\Ysink);
        \fill (\k,\Ysink) circle (2.2pt)
        node[awendpoint,above left=2pt] {\(U_{\k}\)};
      }

      \foreach \r in {0,...,\LastColumn} {
        \pgfmathtruncatemacro{\Smax}{2*\r}
        \foreach \s in {-1,...,\Smax} {
          \draw[awstep] (\r,\s) -- (\r+1,\s);
        }
      }

      \foreach \r in {0,...,\LastColumn} {
        \foreach \s in {0,...,\r} {
          \pgfmathtruncatemacro{\ind}{\r-\s}
          \pgfmathtruncatemacro{\Xdecoration}{\N-\r+2*\s-1}
          \node[above,inner sep=1pt] at (\r+0.5,2*\s)
            {\(X_{\ind,\Xdecoration}\)};
          \ifnum\s=0
            \node[above,inner sep=1pt] at (\r+0.5,-1)
              {\(Y^{(\N)}_{\r}\)};
          \else
            \pgfmathtruncatemacro{\Ydecoration}{\Xdecoration-1}
            \node[above,inner sep=1pt] at (\r+0.5,2*\s-1)
              {\(Y_{\ind,\Ydecoration}\)};
          \fi
        }
      }
      
      \foreach \i in {0,...,\N} {
        \fill (\i,-1) circle (2.2pt)
        node[awendpoint,below=2pt] {\(V_{\i}\)};
      }
    \end{scope}
    \node at (5, 3.5) {\( \Longrightarrow \)};
  \end{tikzpicture}
  \caption{The modification of the weights when \(n=4\).}
  \label{fig:first_lift_2_vertical}
\end{figure}

As before, once the modified weights \( X_{i,\ell} \) and
\( Y_{i,\ell} \) are chosen, the weights \( Y_r^{(n)} \) are then
uniquely determined:
\[
  Y_r^{(n)}
  =\frac{[r+1]_q}{[n-r]_q}(\xi\alpha+\gamma q^r)
    \bigl(q^{n-r-1}+[n-r-1]_q(\beta+\delta q^{n-r-1})\bigr).
\]
Finally, this provided the following conjectural formula:
\begin{equation}\label{eq:original-conj}
\widetilde{\nu}^{Z}_{n,k} = \sum_{\pi\in \mathcal{P}_S(U_{k}\to V_n)}\wt(\pi),
\end{equation}
where \( \wt(\pi) \) is the product of the weights of the east steps in \( \pi \) and the
weight of the east step \( (r,s)\to (r+1,s) \) is given by
\[
  \wt\bigl((r,s)\to(r+1,s)\bigr)
  =\begin{cases}
    X_{r-t,n-r+2t-1}, & \text{if \(s=2t+1\ge1\)},\\
    Y_{r-t-1,n-r+2t}, & \text{if \(s=2t+2\ge2\)},\\
    Y_r^{(n)}, & \text{if \(s=0\)}.
  \end{cases}
\]

The formula \eqref{eq:original-conj} is equivalent to the formula in
our main theorem (\Cref{thm:ASEP-positivity}). To see this, let
\( V'_r = (r,1) \). Then, by considering the first vertex \(V_i\) on
the \(x\)-axis visited by a path
\(\pi\in \mathcal{P}_S(U_k\to V_n)\), one can see that
\eqref{eq:original-conj} is equivalent to
\begin{equation}\label{eq:7}
  \widetilde{\nu}^{Z}_{n,k} = \sum_{i=k}^{n} \sum_{\pi\in \mathcal{P}_S(U_{k}\to V'_i)}\wt(\pi)
  Y_{i}^{(n)}Y_{i+1}^{(n)} \dots Y_{n-1}^{(n)}.
\end{equation}

On the other hand, by \eqref{eq:W-ASEP-def}, we have
\begin{equation}\label{eq:XY-coordinate-weights}
  w_n^Z(r,2t)
  =X_{r-t,n-r+2t-1},
  \qquad
  w_n^Z(r,2t+1)
  =Y_{r-t-1,n-r+2t}.
\end{equation}
Hence, translating the paths from \(U_k\) to \(V'_i\) down by one unit
gives the paths from \(u_k\) to \(v_i\) in \eqref{eq:G-ASEP-def}:
\[
  \sum_{\pi\in \mathcal{P}_S(U_{k}\to V'_i)}\wt(\pi)
  =\sum_{\pi\in\mathcal{P}(u_k \to v_i)}\wt_n^Z(\pi) = G_n^Z(i,k).
\]
Moreover, by \eqref{eq:T-ASEP-def}, we have
\[
  T^Z_{n,i}
  = \prod_{r=0}^{n-i-1} \frac{[i+r+1]_q}{[r+1]_q}
    \bigl(q^r+[r]_q(\beta+\delta q^r)\bigr)
    (\xi\alpha+\gamma q^{i+r})
  = Y_{i}^{(n)} Y_{i+1}^{(n)} \dots Y_{n-1}^{(n)} .
\]
Therefore, \eqref{eq:7} is precisely \Cref{thm:ASEP-positivity}.

\bibliographystyle{abbrv}

\end{document}